\documentclass{amsart}
\usepackage[lofdepth, lotdepth]{subfig}
\usepackage{etex}
\usepackage{pstricks,  amssymb}
\usepackage{pict2e}
\usepackage{geometry}
\usepackage{rotating}
\usepackage{pst-node}
\usepackage{amsmath}

\usepackage{mathrsfs}
\usepackage{xypic}
\usepackage[utf8]{inputenc}
\usepackage{tikz}
\usepackage{tikz-cd}
\usetikzlibrary{arrows}
\usetikzlibrary{decorations.pathreplacing,angles,quotes}

\usepackage[numbers]{natbib}

\usepackage{quiver}
\usepackage{hyperref}
\hypersetup{
	  colorlinks,
	  citecolor=black,
	  filecolor=black,
	  linkcolor=black,
	  urlcolor=black
  }
\usepackage[capitalize,nameinlink,noabbrev,nosort]{cleveref}
\usepackage{amsmath,amscd}
\usepackage{youngtab}
\usepackage[boxsize=.3 em]{ytableau}
\usepackage{verbatim}

\numberwithin{equation}{section}

\theoremstyle{plain}
\newtheorem{theorem}{Theorem}[section]

\newtheorem{prop}[theorem]{Proposition}
\newtheorem{cor}[theorem]{Corollary}
\newtheorem{lemma}[theorem]{Lemma}

\theoremstyle{definition}
\newtheorem{defn}[theorem]{Definition}
\newtheorem{example}[theorem]{Example}

\theoremstyle{remark}
\newtheorem{remark}[theorem]{Remark}

\tikzcdset{scale cd/.style={every label/.append style={scale=#1},
    cells={nodes={scale=#1}}}}

\newcommand{\FFlie}{\mathcal{F}_{\mathrm{Lie}}}

\newcommand{\Lie}{{\mathrm {Lie}}}

\newcommand{\Minor}{\operatorname{Minor}}

\newcommand{\Rmin}{{\R_{\min}}}
\newcommand{\Rmininf}{{\R_{{\min},\infty}}}
\newcommand{\Imin}{\mathcal {IM}}

\newcommand{\Toeplitz}{\operatorname{Toep}}
\newcommand{\ToeplitzU}{\operatorname{UToep}}

\newcommand{\RR}{\mathcal R}
\newcommand{\ep}{\varepsilon}

\newcommand{\bigA}{\mathbf A}
\newcommand{\bigB}{\mathbf B}

\newcommand{\inv}{^{-1}}

\newcommand{\R}{\mathcal{R}}

\newcommand{\C}{\mathbb{C}}

\newcommand{\Z}{\mathbb{Z}}

\newcommand{\Odelpos}{{\mathcal O_{>0}}}
                                                                                                                                                                                                                                                                                                                                                                                                                                                                                                                                                                                                                                                                                                                                                                                                                                                                                                                                                                                                                                                                                                                                                                                                                                                                                                                                                                                                                                                                                                                                                                                                                                                                                                                                                                                                                                                                                                                                                                                                                                                                                                                                                                                                                                                                                                                                                                                                                                                                                                                                                                                                                                                                                                                                                                                                                                                                                                                                                                                                                                                                                                                                                                                                                                                                                                                                                                                                                                                                                                                                                                                                                                                                                                                                                                                                                                                                                                                                                                                                                                                                                                                                                                                                                                                         
                                                                                                                                                                                                                                                                                                                                                                                                                                                                                                                                                                                                                                                                                                                                                                                                                                                                                                                                                                                                                                                                                                                                                                                                                                                                                                                                                                                                                                                                                                                                                                                                                                                                                                                                                                                                                                                                                                                                                                                                                                                                                                                                                                                                                                                                                                                                                                                                                                                                                                                                                                                                                                                                                                                                                                                                                                                                                                                                                                                                                                                                                                                                                                                                                                                                                                                                                                                                                                                                                                                                                                                                                                                                                                                                                                                                                                                                                                                                                                                                                                                                                                                                                                                                                                                         \newcommand{\Kdel} {{\mathcal C}}                                                                                                                                                                                                                                                                                                                                                                                                                                                                                                                                                                                                                                                                                                                                                                                                                                                                                                                                                                                                                                                                                                                                                                                                                                                                                                                                                                                                                                                                                                                                                                                                                                                                                                                                                                                                                                                                                                                                                                                                                                                                                                                                                                                                                                                                                                                                                                                                                                                                                                                                                                                                                                                                                                                                                                                                                                                                                                                                                                                                                                                                                                                                                                                                                                                                                                                                                                                                                                                                                                                                                                                                                                                                                                                                                                                                                                                                                                                                                                                                                                                                                                                                                                                                                                      \newcommand{\Kdelnn} {{\mathcal C_{\ge 0}}}

                                                                                                                                                                                                                                                                                                                                                                                                                                                                                                                                                                                                                                                                                                                                                                                                                                                                                                                                                                                                                                                                                                                                                                                                                                                                                                                                                                                                                                                                                                                                                                                                                                                                                                                                                                                                                                                                                                                                                                                                                                                                                                                                                                                                                                                                                                                                                                                                                                                                                                                                                                                                                                                                                                                                                                                                                                                                                                                                                                                                                                                                                                                                                                                                                                                                                                                                                                                                                                                                                                                                                                                                                                                                                                                                                                                                                                                                                                                                                                                                                                                                                                                                                                            \newcommand{\Kdelpos}{{\mathcal C_{>0}}}
 \newcommand{\Kdelposst}{{\mathcal C_{>0}^{\rm{st}}}}
                                                                                                                                                                                                                                                                                                                                                                                                                                                                                                                                                                                                                                                                                                                                                                                                                                                                                                                                                                                                                                                                                                                                                                                                                                                                                                                                                                                                                                                                                                                                                                                                                                                                                                                                                                                                                                                                                                                                                                                                                                                                                                                                                                                                                                                                                                                                                                                                                                                                                                                                                                                                                                                                                                                                                                                                                                                                                                                                                                                                                                                                                                                                                                                                                                                                                                                                                                                                                                                                                                                                                                                                                                                                                                                                                                                                                                                                                                                                                                                                                                                                                                                                                                                                                                                            \newcommand{\Kdelposwk}{{\mathcal C_{>0}^{\rm{wk}}}}

                                                                                                                                                                                                                                                                                                                                                                                                                                                                                                                                                                                                                                                                                                                                                                                                                                                                                                                                                                                                                                                                                                                                                                                                                                                                                                                                                                                                                                                                                                                                                                                                                                                                                                                                                                                                                                                                                                                                                                                                                                                                                                                                                                                                                                                                                                                                                                                                                                                                                                                                                                                                                                                                                                                                                                                                                                                                                                                                                                                                                                                                                                                                                                                                                                                                                                                                                                                                                                                                                                                                                                                                                                                                                                                                                                                                                                                                                                                                                                                                                                                                                                                                                                                                                                                                                                                                                                                                                                                                                                                                                                                                                                                                                                                                                                                                                                                                                                                                                                                                                                                                                                                                                                                                                                                                                                                                                                                                                                                                                                                                                                                                                                                                                                                                                                                                                                                                                                                                                                                                                                                                                                                                                                                                                                                                                                                                                                                                                                                                                                                                                                                                                                                                                                                                                                                                                                                                                                                                                                                                                                                                                                                                                                                                                                                                                                                                                                                                                                                                                                                                                                                                                                                                                                                                                                                                                                                                                                                                                                                                                                                                                                                                                                                                                                                                                                                                                                                                                                                                                                                                                                                                                                                                                                                                                                                                                                                                                                                                                                                                                                                                                                                                                                                                                                                                                                                                                                                                                                                                                                                                                                                                                                                                                                                                                                                                                                                                                                                                                                                                                                                                                                                                                                                                                                                                                                                                                                                                                                                                                                                                                                                                                                                                                                                                                                                                                                                                                                                                                                                                                                                                                                                                                                                                                                                                                                                                                                                                                                                                                                                                                                                                                                                                                                                                                                                                                                                                                                                                                                                                                                                                                                                                                                                                                                                                                                                                                                                                                                                                                                                                                                                                                                                                                                                                                                                                                                                                                                                                                                                                                                                                                                                                                                                                                                                                                                                                                                                                                                                                                                                                                                                                                                                                                                                                                                                                                                                                                                                                                                                                                                                                                                                                                                                                                                                                                                                                                                                                                                                                                                                                                                                                                                                                                                                                                                                                                                                                                                                                                                                                                                                                                                                                                                                                                                                                                                                                                                                                                                                                                                                                                                                                                                                                                                                                                                                                                                                                                                                                                                                                                                                                                                                                                                                                                                                                                                                                                                                                                                                                                                                                                                                                                                                                                                                                                                                                                                                                                                                                                                                                                                                               \newcommand{\boldalpha}{\boldsymbol{\alpha}}                                                                                                                                                                                                                                                                                                                                                                                                                                                                                                                                                                                                                                                                                                                                                                                                                                                                                                                                                                                                                                                                                                                                                                                                                                                                                                                                                                                                                                                                                                                                                                                                                                                                                                                                                                                                                                                                                                                                                                                                                                                                                                                                                                                                                                                                                                                                                                                                                                                                                                                                                                                                                                                                                                                                                                                                                                                                                                                                                                                                                                                                                                                                                                                                                                                                                                                                                                                                                                                                                                                                                                                                                                                                                                                                                                                                                                                                                                                                                                                                                                                                                                                                                                                                                         
\newcommand{\boldbeta}{\boldsymbol{\beta}}

                                                                                                                                                                                                                                                                                                                                                                                                                                                                                                                                                                                                                                                                                                                                                                                                                                                                                                                                                                                                                                                                                                                                                                                                                                                                                                                                                                                                                                                                                                                                                                                                                                              \newcommand{\N}{\mathbb N} 
                                                                                                                                                                                                                                                                                                                                                                                                                                                                                                                                                                                                                                                                                                                                                                                                                                                                                                                                                                                                                                                                                                                                                                                                                                                                                                                                                                                                                                                                                                                                                                                                                                             \renewcommand{\R}{\mathbb R}

\newcommand{\KK}{\mathbf K}

\newcommand{\Val}{\operatorname{Val}}
\newcommand{\Trop}{\operatorname{Trop}}

\renewcommand{\j}{{\mathrm j}}

\tikzset{labl/.style={anchor=south, rotate=90, inner sep=.5mm}}

\title{Totally positive Toeplitz matrices: classical and modern}
\author{Konstanze Rietsch}
\thanks{This work was partially supported by EPSRC grant EP/V002546/1}

\begin{document}

\begin{abstract} 
By a theorem of Edrei, 
an infinite, normalised totally nonnegative upper-triangular Toeplitz matrix is determined by a pair of nonnegative  parameter sequences $(\boldalpha,\boldbeta)$, the `Schoenberg parameters', where nonzero parameters correspond to the roots and poles of a naturally associated generating function. These totally nonnegative Toeplitz matrices and their parameters also arise in the classification of characters of the infinite symmetric group $S_\infty$ by later work of Thoma \cite{Thoma}. Moreover the Schoenberg parameters have an asymptotic interpretation in terms of irreducible representations of $S_n$ and their Young diagrams by Vershik-Kerov \cite{VershikKerovAsymptotic}. In this article we consider infinite totally positive Toeplitz matrices as limits of finite ones, and we obtain two further asymptotic descriptions of the Schoenberg parameters that are now related to quantum cohomology of the flag variety $\mathcal Fl_{n+1}$ as $n\to\infty$. One is related to asymptotics of normalised quantum parameters, and the other to asymptotics of the Chern classes of the tautological line bundles. We also describe the asymptotics of (quantum) Schubert classes in terms of $(\boldalpha,\boldbeta)$. 
Our limit formulas relate to and were motivated by a tropical analogue of this theory \cite{rietschToeplitz} that we survey. In the tropical setting one finds an asymptotic relationship between the `tropical Schoenberg parameters' and the weight map from Lusztig's parametrisation of the canonical basis.
 \end{abstract}

\maketitle

\setcounter{tocdepth}{2}
%\tableofcontents

 \section{Introduction}%: Three 20th century theorems on infinite Toeplitz matrices}
 \label{s:20thCentury}

The theory of total positivity, concerning itself with matrices with nonnegative  minors, was pioneered around 100 years ago by work of P\'olya, Fekete, Schoenberg and Gantmacher-Krein, see \cite{Karlin,GK:oscillation}. It  has two eras roughly, the early era before the mid 1990's and  the later one, after Lusztig's work \cite{Lusztig94} on total positivity and canonical bases reinvigorated the field which has now grown very vast. For our introduction, we start in the earlier period and recall three theorems which are the theorems of Edrei (1952), Thoma (1964) and Vershik-Kerov (1982). Then we will introduce more modern, but related topics as they arise. 

\subsection{The Edrei Theorem}\label{s:EdreiThm} 
One of the highlights of the classical theory of total positivity is the proof, completed by Edrei,  of Schoenberg's  conjecture on the structure of totally nonnegative infinite upper-triangular Toeplitz matrices. This conjecture, made by Schoenberg in 1948 in honour of  Courant's $60$\textsuperscript{th} birthday \cite{Schoenberg:48}, can be thought of as parametrising such matrices via factorisations of their associated generating functions. Let us denote by $u(\mathbf c)$ the infinite Toeplitz matrix corresponding to the sequence $\mathbf c=(1,c_1,c_2, c_3,\dotsc)$, 
\begin{equation}\label{e:u}
u(\mathbf c)=\begin{pmatrix}
1 &c_1 & c_2& c_3& c_4 &  & \quad\\
 & 1  & c_1& c_2 & c_3& \ddots  &\\
  &    & 1  & c_1 & c_2 &\ddots &\\
  &    &    &1 &c_1 &\ddots &\\
  &    &    &       &  \ddots  &\ddots &  \\
  &&&&&& 
\end{pmatrix},
\end{equation}
and call $u(\mathbf c)$ a \textit{totally nonnegative matrix} and the sequence $\mathbf c$ a \textit{totally nonnegative sequence} if all of the minors of $u(\mathbf c)$ lie in $\R_{\ge 0}$. Such a sequence is also often referred to as a P\'olya frequency sequence. 
\begin{theorem} [Edrei Theorem {\cite{ASW,Edrei52}}]\label{t:EdreiIntro}
A sequence $\mathbf c=(c_i)_{i=1}^\infty$ of real numbers is totally nonnegative if and only if its generating function has a factorisation of the form
\begin{equation}\label{e:SchoenbergFormula}
1+c_1x+c_2x^2+c_3x^3+\dotsc =e^{\gamma x}\prod_{i=1}^\infty\frac{1+\beta_ix}{1-\alpha_ix},
\end{equation}
for some $((\boldalpha,\boldbeta),\gamma)\in\Omega_S\times\R_{\ge 0}$ where $\Omega_S$ denotes the parameter space
\[
\Omega_{S}=\left\{(\boldalpha,\boldbeta)\in\R_{\ge 0}^{\N}\times \R_{\ge 0}^\N\left|\begin{array}{l} \boldalpha=(\alpha_i)_{i\in\N} \text{ with } \alpha_{i}\ge \alpha_{i+1} \text{ and }\sum_{i}\alpha_i<\infty\\ \boldbeta=(\beta_j)_{j\in\N} \text{ with }  \beta_{j}\ge \beta_{j+1} \text{ and }\sum_{j}\beta_j<\infty\end{array}\right.\right\}. 
\]
\end{theorem}
\begin{remark}\label{r:gamma} We note that
% it is immediate 
for any totally nonnegative Toeplitz matrix $u(\mathbf c)$ as in Theorem~\ref{t:EdreiIntro} the sum $\sum_i(\alpha_i+\beta_i)$ equals to  $c_1$ if and only if $\gamma=0$, and is in any case bounded above by $c_1$. For the proof of this theorem, it was initially shown by Schoenberg that \eqref{e:SchoenbergFormula} defines a totally nonnegative sequence, and  together with M.~Aissen and A.~Whitney \cite{ASW} that conversely the roots and poles of a totally nonnegative sequence are as described. The key and most mysterious part of Theorem~\ref{t:EdreiIntro}, which was proved by Edrei in \cite{Edrei52}, is that the factor that remains, necessarily an entire function with no zeros, is of the simple form~$e^{\gamma x}$. 
 \end{remark}
\subsection{The Thoma Theorem}\label{s:ThomaThm}
Apart from the intrinsic beauty of Theorem~\ref{t:EdreiIntro}, infinite totally nonnegative Toeplitz matrices came up again independently a decade later and were shown to have deep connections to representation theory, further adding to their interest. Namely, Thoma \cite{Thoma} reproved Theorem~\ref{t:EdreiIntro}, unaware of Edrei's earlier work, while giving his parametrisation of extremal characters of the infinite symmetric group~$S_\infty$. We refer the reader to  \cite{BorodinOlshanskiBook} for a modern introduction to  this theory as well as the standard notations. In the statement of Thoma's theorem below let $\bar \chi_\lambda=\frac{\chi_\lambda}{\chi_\lambda(e)}$ denote the normalised character of $S_n$ associated to a partition $\lambda=(\lambda_1,\dotsc, \lambda_r)$ of $n$, and let $\Delta_\lambda(u)$ denote the Toeplitz minor $\det(c_{\lambda_i+\j-i})_{i,j=1}^r$ of $u(\mathbf c)$ from \eqref{e:u} associated to the same partition. 
The trivial character of $S_n$ is $\chi_{(n)}$ in this notation. Recall also that for any class function $\chi$ on $S_n$ the inner product $\langle \chi,\chi_{(n)}\rangle$ is just the average value of $\chi$.

\begin{theorem}[Thoma Theorem~\cite{Thoma}]
\label{t:ThomaIntro} The extremal characters of $S_\infty$ are precisely the multiplicative class functions $\chi:S_\infty\to \C$ for which  the associated sequence $\mathbf c=(1,c_2,c_3,\dotsc)$ defined by $c_n=\langle\chi|_{S_n},\chi_{(n)}\rangle$ is totally nonnegative. 
Moreover, the restriction of $\chi$ to $S_n$ is then the convex combination% of normalised irreducible characters,
\[
\chi|_{S_n}=\sum_{\lambda\vdash n}\Delta_\lambda(u(\mathbf c))\,\bar\chi_\lambda,
\]
of (normalised) irreducible characters of $S_n$, and $\chi$ is given by  the explicit formula
\[
\chi((k_1+1,\dotsc, k_2)(k_3+1,\dotsc ,k_4)\dotsc (k_{2m-1}+1,\dotsc, k_{2m}))=\prod_{\ell=1}^m \left(\sum_{i=1}^\infty\alpha_i^{k_{2\ell}-k_{2\ell-1}}-(-\beta_i)^{k_{2\ell}-k_{2\ell-1}}\right)
\]  
for any product of nontrivial disjoint cycles, using the parameters from Theorem~\ref{t:EdreiIntro}.
\end{theorem}
\begin{remark} Note that because of the $c_1=1$ normalisation in Thoma's theorem the parameter space for the extremal characters of $S_\infty$ is the subset of $\Omega_S$ defined by
\[
\Omega_T=\left\{(\boldalpha,\boldbeta)\in\R_{\ge 0}^{\N}\times \R_{\ge 0}^\N\left|\begin{array}{l} \boldalpha=(\alpha_i)_{i\in\N} \text{ with } \alpha_i\ge \alpha_{i+1} \\ \boldbeta=(\beta_j)_{j\in\N} \text{ with }  \beta_{j}\ge \beta_{j+1}\\ \sum_{k=1}^\infty \alpha_k+\beta_k\le 1\end{array}\right.\right\},
\]
where we set $\gamma=1-\sum_{k=1}^\infty (\alpha_k+\beta_k)$ in \eqref{e:SchoenbergFormula}. This is a compact metrisable space sometimes called the Thoma simplex. The multiplicativity of $\chi$ means that $\chi$ is in fact determined by its values on nontrivial $k$-cycles, $\chi((1\dotsc k))=\sum_{i=1}^\infty\left(\alpha_i^k + (-1)^{k+1}\beta_i^k\right)$. We may henceforth refer to the  parameters $(\boldalpha,\boldbeta)$ from $\Omega_S$ as Schoenberg parameters and the parameters in $\Omega_T$ as Thoma parameters.
\end{remark}

\subsection{The Vershik-Kerov Theorem}\label{s:VKThm} Various proofs of the Edrei and Thoma Theorems have been given since the original ones, see \cite{VershikKerovAsymptotic,Olshanskii90,OkThoma,
OkInfSymmGrp,KOO-IMRN,GohmKostler12,BufetovGorinThoma,PetrovThoma,VT:SchurWeyl} see also \cite{Kerov-thesis,VershikSurvey}.  The proof of Vershik and Kerov \cite{VershikKerovAsymptotic} involves an asymptotic interpretation of the Thoma parameters that adds another intriguing dimension to the picture.  Let us describe a partition $\lambda\vdash n$ via its modified Frobenius coordinates $(a|b)$, where $a$ and $b$ are the sequences of `arm-' and `leg-lengths',  $a_i=\lambda_i-i+\frac 12$ and $b_j=\lambda'_j-j+\frac 12$ of the corresponding Young diagram. Here $\lambda'$ is the transpose of $\lambda$, and $i,j=1,\dotsc,m(\lambda)$ with $m(\lambda)$ equal to the number of diagonal boxes of $\lambda$. 
\begin{theorem}[{Vershik-Kerov Theorem \cite{VershikKerovAsymptotic}}] 
For any sequence of partitions $(\lambda^{(n)})_{n=1}^\infty$ with $\lambda^{(n)}\vdash n$ described by their modified  Frobenius coordinates $(a^{(n)}|b^{(n)})$ the following are equivalent.
\begin{itemize}
\item The sequences $\left(a_i^{(n)}/n\right)_n$, and $\left(b_j^{(n)}/n\right)_n$ converge as $n\to\infty$.  
\item For any $\sigma\in S_{\infty}$ the sequence $\left(\bar\chi_{\lambda^{(n)}}(\sigma)\right)_n$  of normalised character values converges as $n\to\infty$. 
\end{itemize}
In this case, the class function $\chi:S_{\infty}\to\R$ defined by $\chi(\sigma)=\lim_{n\to\infty}\bar\chi_{\lambda^{(n)}}(\sigma)$ is an extremal character of~$S_\infty$ and every extremal character arises in this way. Moreover, the limits
\begin{equation}
\alpha_i:=\lim_{n\to\infty}\frac{a_i^{(n)}}n \quad\text{ and }\quad \beta_j:=\lim_{n\to\infty}\frac{b_j^{(n)}}n
\end{equation} 
are the Thoma parameters of $\chi$. 
\end{theorem}

While we will only be concerned with upper-triangular matrices and one-sided sequences, we mention that both Edrei and Thoma's theorems have analogues for two-sided sequences \cite{EdreiDouble,Voiculescu} with $S_\infty$ replaced by $U(\infty)$, as well as there existing versions for other infinite classical types, e.g. \cite{VoiculescuClassical,HiraiHirai}. 
Another interesting variant is the continuous analogue of so-called `P\'olya frequency functions' also advanced by Schoenberg, see \cite{Grochenig23} and references therein. A modern generalisation parallel to the one considered here relates to loop groups \cite{LP:I,LP:II}, and  a further interesting recent  direction  involves working over $\N$  \cite{BorgerGrinberg,HoHai,Davydov}.

\section{Finite totally positive Toeplitz matrices and asymptotics} \label{s:MainFiniteReal}
  As a consequence of Thoma's theorem, the classification of irreducible representations of $S_n$ is a natural finite analogue of the parametrisation of totally nonnegative Toeplitz matrices. %$u(\mathbf c)$, and the asymptotic relationship between these two theories was uncovered by Vershik and Kerov. 
  A different natural finite analogue would be the problem of  parametrising  finite totally nonnegative upper-triangular Toeplitz matrices. This problem was solved fifty years after Edrei's theorem, in  \cite{rietschJAMS,rietschNagoya}. Moreover, the finite theory has its own separate context, relating to flag varieties, quantum cohomology and mirror symmetry. 

\subsection{A parametrisation theorem for finite Toeplitz matrices}\label{s:FiniteParam}
Let $U_+=U_+^{(n+1)}$ denote the unipotent upper-triangular matrices and $T=T_{SL_{n+1}}$ the diagonal torus in $SL_{n+1}$ with Lie algebra $\mathfrak h_{SL_{n+1}}$.  We will now focus on total \textit{positivity}, by which we mean that all non-trivial minors should be strictly positive. 
Let $U_+(\R_{>0})$ denote the totally positive part of $U_+$ and $\ToeplitzU_{n+1}(\R_{>0})$ its subset of totally positive %$(n+1)\times (n+1)$ 
Toeplitz matrices.  Note these are both open dense in their totally nonnegative counterparts. % $\ToeplitzU_{n+1}(\R_{>0})$ is an open dense subset of its totally nonnegative counterpart  $\ToeplitzU_{n+1}(\R_{\ge 0})$. 
We let $\Delta_i(u)$ denote the minor with row set $[1,i]$ and  column set $[n+2-i,n+1]$ of an $(n+1)\times (n+1)$ matrix $u$.   
\begin{theorem}[{\cite{rietschJAMS,rietschNagoya}}]\label{t:IntroRfin}
We have a homeomorphism
\begin{eqnarray}\label{e:R03}
\ToeplitzU_{n+1}(\R_{>0})&\longrightarrow & T_{SL_{n+1}}(\R_{>0})
\end{eqnarray}
defined by sending $u\in \ToeplitzU_{n+1}(\R_{>0})$ to the  matrix $\Delta(u)\in T_{SL_{n+1}}(\R_{>0})$ with diagonal entries
\begin{equation}\label{e:di}
d_1={\Delta_{1}(u)},\   
d_2=\frac{\Delta_{2}(u)}{\Delta_{1}(u)},\ 
\dotsc,\ d_n=\frac{\Delta_{n}(u)}{\Delta_{{n-1}}(u)},\ d_{n+1}=\frac{1}{\Delta_{n}(u)}.
\end{equation}
\end{theorem}
Composing \eqref{e:R03} with the homeomorphism $T_{SL_{n+1}}(\R_{>0})\overset \sim\longrightarrow \R_{>0}^{n}$ which sends the diagonal matrix $(\delta_{ij}d_i)_{i=1}^{n+1}$ to $(\frac{d_{2}}{d_1},\frac{d_3}{d_2},\dotsc,\frac{d_{n+1}}{d_n})$ furthermore defines a homeomorphism
\begin{eqnarray}\label{e:q-paramToeplitz}
\mathbf q=(q_1,\dotsc, q_n):\ToeplitzU_{n+1}(\R_{>0})&\longrightarrow & \R_{>0}^{n},
\end{eqnarray}
that can be considered as giving a parametrisation of $\ToeplitzU_{n+1}(\R_{>0})$ with $n$ positive parameters. 

The first proof of this parametrisation theorem involved quantum cohomology of flag varieties, positivity of Gromov-Witten invariants and work of Dale Peterson \cite{peterson, Kos:QCoh}. Here the positive parameters are naturally interpreted as quantum parameters and related to curve class  generators for  the flag variety $\mathcal Fl_{n+1}=SL_{n+1}/B$  \cite{rietschJAMS}. Then in \cite{rietschNagoya,rietsch} another approach to the same result and its $SL_{n+1}/P$ generalisation was given that involved mirror symmetry and took inspiration from early Laurent polynomial mirror constructions for flag varieties.
See also Section~\ref{s:QCoh}. We now take a particular perspective on $\ToeplitzU_{n+1}(\R_{>0})$ as subset of $U_+(\R_{>0})$. We use standard notations from algebraic groups, e.g.~\cite{Humphreys:LAG,Springer:book}.
\subsection{Toeplitz matrices in standard coordinates}\label{s:StandardCoord}
Lusztig in \cite{Lusztig94} described  $U_+(\R_{>0})$ in terms of specific coordinate charts corresponding  to factorisations of elements into simple root subgroups $x_k(m)=\exp(m e_k)$.  
We choose a particular such positive chart. 
Namely, consider an ordered  index set $\mathcal S_{\le n+1}=\{(i,j)\mid i+j\le n+1\}$ with $(i,j)<(i',j')$ if $i>i'$ and $(i,j)<(i,j')$ if $j<j'$. Then any $u\in U_+(\R_{>0})$ is an ordered~product
\[
u=\prod_{(i,j)\in \mathcal S_{\le n+1}}x_{i+j-1}(m_{ij})
\]
for unique positive coordinates $(m_{ij})_{i+j\le n+1}$. 
For example, for $n=3$ 
%we have coordinates $(m_{ij})_{i+j\le 4}$  and 
the factorisation of $u\in U_+(\R_{>0})$ can be read off conveniently row by row from the diagram below, which gives the matrix shown to its right, 
\begin{equation*}%\begin{center}
\begin{tikzpicture}[scale=1.1]
\node[fill=white] at    (0.5,2.55) {$x_3(m_{31})$}  ;
\node[fill=white] at    (0.5,1.55) {$x_2(m_{21})$}  ;
\node[fill=white] at    (0.5,0.55) {$x_1(m_{11})$}  ;
\node[fill=white] at    (1.8,1.55) {$x_3(m_{22})$}  ;\node[fill=white] at    (1.8,0.55) {$x_2(m_{12})$}  ;
\node[fill=white] at    (3,0.55) {$x_3(m_{13})$}  ;
\node[fill=white] at (5.5,1.5) {} ;
\node[fill=white] at (8.5,1.5) { $u((m_{ij}))=
\begin{pmatrix}1 & m_{11}& m_{11}m_{12} & m_{11} m_{12} m_{13}\\
0 &1& m_{21}+m_{12} & m_{21}m_{22}+m_{21}m_{13}+m_{12}m_{13}  \\
0 & 0& 1& m_{31}+m_{22}+m_{13} \\
0&0&0& 1
\end{pmatrix}$.} ;
\draw[black] (-.2,0.1) -- (-.2,3.1) -- (1.15,3.1) -- (1.15,0.1) -- (3.6,0.1) -- (3.6,1.1)--(-.2,1.1)--(-.2,0.1)--(2.4,0.1)--(2.4,2.1)--(-.2,2.1);
 \end{tikzpicture},
%\end{center}
\end{equation*}
We write $u=u((m_{ij}))$ and call the $m_{ij}$  \textit{standard coordiates}. This chart immediately has two key properties. 
\begin{itemize}
\item It parametrises the totally positive part $U_+(\R_{>0})$ by positive parameters, see \cite{Lusztig94}, with an explicit algebraic inverse given in terms of minors, see Equation~\ref{e:mviaminors}.
\item It extends in a natural way to the projective limit $U_+^{(\infty)}$ giving a standard coordinate system $(m_{ij})_{i,j\in\N}\in \R_{>0}^{\N\times\N}$ for the  infinite totally positive matrices $U_+^{(\infty)}(\R_{>0})$. 
\end{itemize}
The third key property will be that the totally positive Toeplitz matrices (finite and infinite) have a convenient description. Namely, the Toeplitz condition imposes relations on the standard coordinates that can be expressed as a kind of multiplicative `divergence-free' property for a labelled quiver.

\begin{defn}\label{d:divergence-free} Consider a connected quiver $Q$. Assume further that $Q$ has precisely one source vertex, to which we assign the label $0$. By a (positive) vertex labelling of $Q$, we mean an assignment of a (positive) real number to every non-source vertex. Given such a labelling we also label each arrow by the vertex label at its head minus the vertex label at its tail. We call a labelling  `\textit{divergence-free}' if for every vertex that is neither a sink nor a source, the product of the incoming arrow labels equals to the product of the outgoing ones. We may also replace $\R$ and $\R_{>0}$ in this definition by any other ring $\mathcal R$ and positive subsemifield $\RR_{>0}$.
\end{defn}

\begin{defn}\label{d:Qn+1}
Recall the indexing set $\mathcal S_{\le n+1}$. We 
define a quiver $Q_{n+1}=(\mathcal V_{n+1},\mathcal A_{n+1})$ with vertex set $\mathcal V_{n+1}=\mathcal S_{\le n+1}\cup\{(0,0)\}$ and arrows $\mathcal A_{n+1}=\{a_{ij}\mid (i,j)\in \mathcal S_{\le n+1}\}\sqcup \{a'_{ij}\mid (i,j)\in \mathcal S_{\le n+1}\}$.  Here the index of an arrow signifies its head, $h(a_{ij})=h(a'_{ij})=(i,j)$. The arrows $a_{ij}$  have tail $(i-1,j)$, or if  $i=1$ then  $(0,0)$, and are called `vertical'. The arrows $a'_{ij}$ have tail $(i,j-1)$, or if $j=1$ then $(0,0)$, and are called `horizontal'. 
The quiver $Q_\infty$ is defined similarly, with $S_{\le n+1}$ replaced by $\N\times\N$. See
%We embed these quivers into the plane as shown in 
Figure~\ref{f:Qn+1} and Figure~\ref{f:QinfQ5}.
\end{defn}

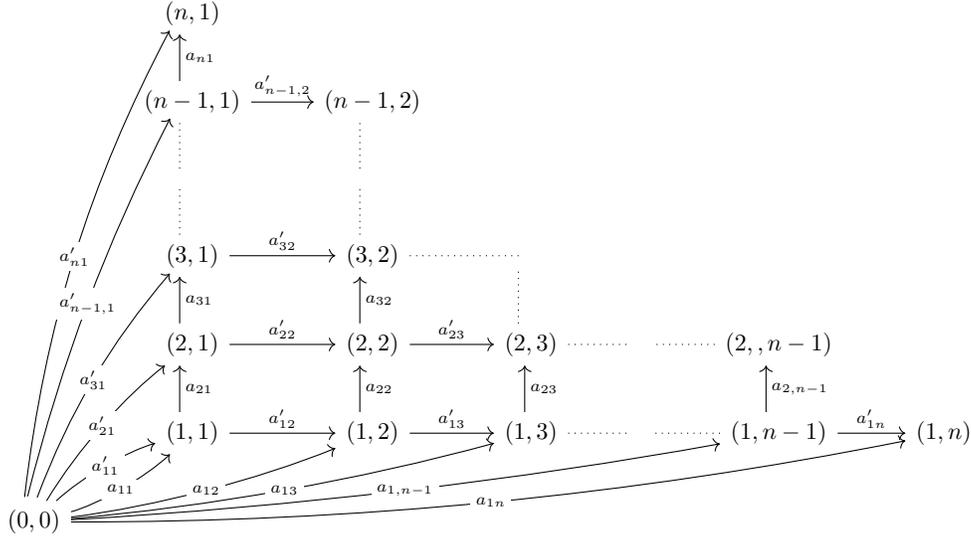
\begin{figure}
\[\begin{tikzcd}[scale cd=0.9]
	& {(n,1)} \\
	& {(n-1,1)} & {(n-1,2)} \\
	& {} & {} & {} \\
	& {(3,1)} & {(3,2)} & {} & {} \\
	& {(2,1)} & {(2,2)} & {(2,3)} & {} & {(2,,n-1)} \\
	& {(1,1)} & {(1,2)} & {(1,3)} & {} & {(1,n-1)} & {(1,n)} \\
	{(0,0)}
	\arrow["{a_{n1}}"', shift left=2, from=2-2, to=1-2]
	\arrow["{a'_{n-1,2}}", from=2-2, to=2-3]
	\arrow[shift right=2, dotted, no head, from=2-2, to=3-2]
	\arrow[shift right=2, dotted, no head, from=2-3, to=3-3]
	\arrow[shift left=2, dotted, no head, from=4-2, to=3-2]
	\arrow["{a'_{32}}", from=4-2, to=4-3]
	\arrow[shift left=2, dotted, no head, from=4-3, to=3-3]
	\arrow[dotted, no head, from=4-3, to=4-4]
	\arrow["{a_{31}}"', shift left=2, from=5-2, to=4-2]
	\arrow["{a'_{22}}", from=5-2, to=5-3]
	\arrow["{a_{32}}"', shift left=2, from=5-3, to=4-3]
	\arrow["{a'_{23}}", from=5-3, to=5-4]
	\arrow[shift left=2, dotted, no head, from=5-4, to=4-4]
	\arrow[dotted, no head, from=5-4, to=5-5]
	\arrow[dotted, no head, from=5-6, to=5-5]
	\arrow["{a_{21}}"', shift left=2, from=6-2, to=5-2]
	\arrow["{a'_{12}}", from=6-2, to=6-3]
	\arrow["{a_{22}}"', shift left=2, from=6-3, to=5-3]
	\arrow["{a'_{13}}", from=6-3, to=6-4]
	\arrow["{a_{23}}"', shift left, from=6-4, to=5-4]
	\arrow[dotted, no head, from=6-4, to=6-5]
	\arrow["{a_{2,n-1}}"', shift left=2, from=6-6, to=5-6]
	\arrow[dotted, no head, from=6-6, to=6-5]
	\arrow["{a'_{1n}}", from=6-6, to=6-7]
	\arrow["{a'_{n1}}"{description}, shift left=2, curve={height=-12pt}, from=7-1, to=1-2]
	\arrow["{a'_{n-1,1}}"{description}, shift left=2, curve={height=-6pt}, from=7-1, to=2-2]
	\arrow["{a'_{31}}"{description}, shift left, curve={height=-6pt}, from=7-1, to=4-2]
	\arrow["{a'_{21}}"{description}, curve={height=-6pt}, from=7-1, to=5-2]
	\arrow["{a'_{11}}"{description}, curve={height=-6pt}, from=7-1, to=6-2]
	\arrow["{a_{11}}"{description}, curve={height=6pt}, from=7-1, to=6-2]
	\arrow["{a_{12}}"{description}, curve={height=6pt}, from=7-1, to=6-3]
	\arrow["{a_{13}}"{description}, curve={height=6pt}, from=7-1, to=6-4]
	\arrow["{a_{1,n-1}}"{description}, curve={height=6pt}, from=7-1, to=6-6]
	\arrow["{a_{1n}}"{description}, curve={height=12pt}, from=7-1, to=6-7]
\end{tikzcd}\]
\caption{The quiver $Q_{n+1}$\label{f:Qn+1}}
\end{figure}

\begin{defn}\label{d:vij}
Given any $u\in U_+(\R_{>0})$, we construct a positive vertex labelling of $Q_{n+1}$ by associating
\begin{equation}\label{e:vij}
v_{ij}:= \sum_{\ell=0}^{\min(i,j)-1}\frac{1}{m_{i-\ell,j-\ell}}=\frac 1{m_{i,j}}+\frac{1}{m_{i-1,j-1}}+\dotsc.
\end{equation}
to the vertex $(i,j)$, where $(m_{ij})_{i,j}\in(\R_{>0})^{\mathcal S_{\le n+1}}$ are the standard coordinates of $u$. See Figure~\ref{f:QinfQ5} for an example. For an infinite totally positive matrix $u\in U^{(\infty)}_+(\R_{>0})$ we use its standard coordinates $(m_{ij})_{i,j\in\N}$ and the same formula to associate a vertex labelling of the infinite quiver $Q_{\infty}$. 
\end{defn} 

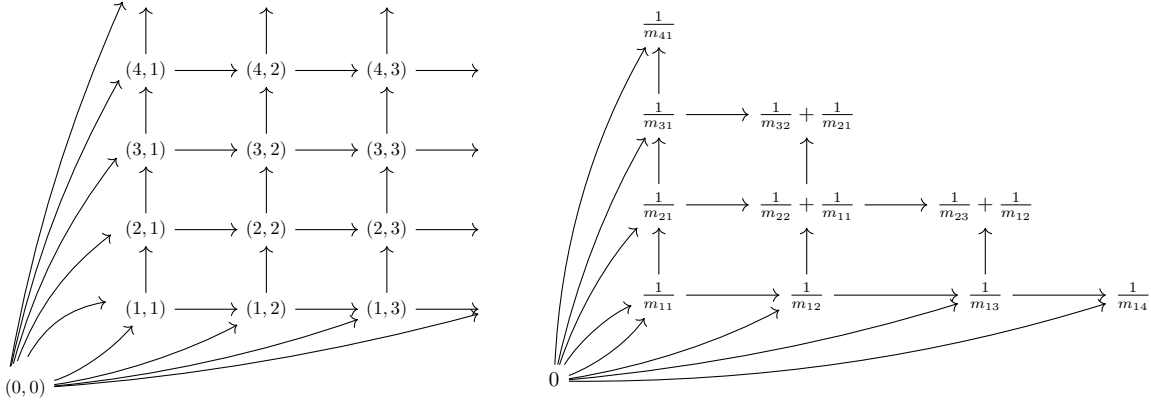
\begin{figure}[h!]
\[\begin{tikzcd}[scale cd=0.7]
	{\quad}& {\quad} 	& {\quad}		& {\quad}& {\quad}	\\
	& { (4,1)} & {(4,2)} &(4,3)  & {\quad} \\
	& { (3,1)} & { (3,2)} & {(3,3)}  &{\quad}\\
	& {(2,1)} & { (2,2)} & { (2,3)}&{\quad} \\
	& {(1,1)} & { (1,2)} & { (1,3)}  &{\quad} \\
	(0,0)
	\arrow[from=2-2, to=1-2]
	\arrow[from=2-2, to=2-3]
	\arrow[from=2-3, to=2-4]
	\arrow[from=2-4, to=2-5]
	%\arrow[from=2-5, to=2-6]
	\arrow[from=3-2, to=2-2]
	\arrow[from=3-2, to=3-3]
	\arrow[from=3-3, to=2-3]
	\arrow[from=3-3, to=3-4]
	\arrow[from=3-4, to=3-5]
	%\arrow[from=3-5, to=3-6]
	\arrow[from=4-2, to=3-2]
	\arrow[from=4-2, to=4-3]
	\arrow[from=4-3, to=3-3]
	\arrow[from=4-3, to=4-4]
	\arrow[from=4-4, to=3-4]
	\arrow[from=4-4, to=4-5]
	\arrow[from=5-2, to=4-2]
	%\arrow[from=4-5, to=4-6]
	\arrow[from=5-2, to=5-3]
	\arrow[from=5-3, to=4-3]
	\arrow[from=5-3, to=5-4]
	\arrow[from=5-4, to=4-4]
	\arrow[from=5-4, to=5-5]
	%\arrow[from=5-5, to=4-5]
	%\arrow[from=5-5, to=5-6]
	\arrow[shift left=3,curve={height=-6pt}, from=6-1, to=1-2]
	\arrow[shift left=3,curve={height=-6pt}, from=6-1, to=2-2]
	\arrow[shift left=3,curve={height=-6pt}, from=6-1, to=3-2]
	\arrow[shift left=3,curve={height=-6pt}, from=6-1, to=4-2]
	\arrow[shift left=3,curve={height=-6pt}, from=6-1, to=5-2]
	\arrow[from=2-3, to=1-3]
	\arrow[from=2-4, to=1-4]
	%\arrow[from=2-5, to=1-5]
	\arrow[from=3-4, to=2-4]
	%\arrow[ from=3-5, to=2-5]	
	%\arrow[from=4-5, to=3-5]	
	\arrow[curve={height=6pt}, from=6-1, to=5-2]
	\arrow[curve={height=6pt}, from=6-1, to=5-3]
	\arrow[curve={height=6pt}, from=6-1, to=5-4]
	\arrow[curve={height=6pt}, from=6-1, to=5-5]
	%\arrow[curve={height=6pt}, from=6-1, to=5-6]
\end{tikzcd}
\quad
% https://q.uiver.app/#q=WzAsMTEsWzEsMywiMS9tX3sxMX0iXSxbMiwzLCIxL21fezEyfSJdLFszLDMsIjEvbV97MTN9Il0sWzEsMiwiMS9tX3syMX0iXSxbMSwxLCIxL21fezMxfSJdLFsyLDIsIlxcZnJhYzF7bV97MjJ9fStcXGZyYWMxe21fezExfX0iXSxbMCw0LCIwIl0sWzIsMSwiXFxmcmFjIDF7bV97MzJ9fStcXGZyYWMgMXttX3syMX19Il0sWzEsMCwiMS9tX3s0MX0iXSxbNCwzLCIxL21fezE0fSJdLFszLDIsIlxcZnJhYyAxe21fezIzfX0rXFxmcmFjIDF7bV97MTJ9fSJdLFswLDFdLFsxLDJdLFswLDNdLFszLDRdLFszLDVdLFsxLDVdLFs2LDAsIiIsMCx7ImN1cnZlIjoxfV0sWzYsMCwiIiwwLHsiY3VydmUiOi0xfV0sWzYsMSwiIiwwLHsiY3VydmUiOjF9XSxbNiwyLCIiLDAseyJjdXJ2ZSI6MX1dLFs2LDMsIiIsMCx7ImN1cnZlIjotMX1dLFs2LDQsIiIsMix7ImN1cnZlIjotMX1dLFs1LDddLFs0LDhdLFs2LDgsIiIsMCx7ImN1cnZlIjotMn1dLFs2LDksIiIsMCx7ImN1cnZlIjoyfV0sWzIsMTBdLFsyLDldLFs1LDEwXSxbNCw3XV0=
\begin{tikzcd}[scale cd=0.8]
	& {\frac 1{m_{41}}} \\
	& {\frac{1}{m_{31}}} & {\frac 1{m_{32}}+\frac 1{m_{21}}} \\
	& {\frac 1{m_{21}}} & {{\frac1{m_{22}}+\frac1{m_{11}}}} & {\frac 1{m_{23}}+\frac 1{m_{12}}} \\
	& {\frac{1}{m_{11}}} & {\frac 1{m_{12}}} & {\frac 1{m_{13}}} & {\frac{1}{m_{14}}} \\
	0
	\arrow[from=2-2, to=1-2]
	\arrow[from=2-2, to=2-3]
	\arrow[from=3-2, to=2-2]
	\arrow[from=3-2, to=3-3]
	\arrow[from=3-3, to=2-3]
	\arrow[from=3-3, to=3-4]
	\arrow[from=4-2, to=3-2]
	\arrow[from=4-2, to=4-3]
	\arrow[from=4-3, to=3-3]
	\arrow[from=4-3, to=4-4]
	\arrow[from=4-4, to=3-4]
	\arrow[from=4-4, to=4-5]
	\arrow[curve={height=-12pt}, from=5-1, to=1-2]
	\arrow[curve={height=-6pt}, from=5-1, to=2-2]
	\arrow[curve={height=-6pt}, from=5-1, to=3-2]
	\arrow[curve={height=6pt}, from=5-1, to=4-2]
	\arrow[curve={height=-6pt}, from=5-1, to=4-2]
	\arrow[curve={height=6pt}, from=5-1, to=4-3]
	\arrow[curve={height=6pt}, from=5-1, to=4-4]
	\arrow[curve={height=12pt}, from=5-1, to=4-5]
\end{tikzcd}\]
\caption{On the left, a segment of $Q_{\infty}$ and on the right the vertex labelling of the quiver $Q_{n+1}$ for $n=4$ from Definition~\ref{d:vij}.\label{f:QinfQ5}}
\end{figure}

We can now describe the totally positive Toeplitz matrices in terms of standard coordinates as follows.

\begin{prop}{{\cite[Proposition A.1 and Corollary~7.6]{rietschToeplitz}}}\label{p:ToeplitzEquiv}  The (finite or infinite) totally positive unipotent upper-triangular matrix $u=u((m_{ij}))$ is \textit{Toeplitz} if and only if its associated quiver labelling of $Q_{n+1}$ (resp. $Q_{\infty}$) from Definition~\ref{d:vij} is divergence-free. The product of the labels of the arrows $a_{i,k+1-i},a'_{i,k+1-i}$ into $(i,k+1-i)$  recovers the $i$\textsuperscript{th} parameter $q_i$ from \eqref{e:q-paramToeplitz} for the truncation  $u^{[k+1]}\in\ToeplitzU_{k+1}(\R_{>0})$ 
of $u$.  
\end{prop}

\begin{remark}\label{r:GiventalComparison} The quiver $Q_{n+1}$ is essentially the planar dual of a quiver introduced by Givental \cite{Givental:QToda}, and this proposition relates to a different construction of upper-triangular and Toeplitz matrices \cite{rietsch} via that quiver together with a conjecture about it proved recently by L\"udenbach  \cite[Theorem~4.3.1]{Ludenbach}. See also Section~\ref{s:ChernRoots}.  However, the connection between the construction of Toeplitz matrices in Proposition~\ref{p:ToeplitzEquiv} and the earlier one is subtle in that our matrix $u((m_{ij}))\in U^{(n+1)}_+$ associated to $Q_{n+1}$, and the alternatively constructed upper-triangular matrix from \cite{rietsch} related to Givental's quiver  
do not  agree, \textit{except} for when $u((m_{ij})_{ij})$ is Toeplitz.  L\"udenbach's theorem is involved, and makes use of the chamber ansatz \cite{FZ:DoubleBruhat} in its proof. The proof of Proposition~\ref{p:ToeplitzEquiv} in \cite[Appendix~A]{rietschToeplitz} is direct and not connected to these earlier~constructions. 
\end{remark}

\subsection{Quantum cohomology of flag varieties}
\label{s:QCoh}
Consider the variety  of flags of linear subspaces of $\C^{n+1}$,
\[
\mathcal Fl_{n+1}=\{V_\bullet=(V_k)_k\mid \{0\}\subset V_1\subset  V_2\subset\cdots  \subset V_{n}\subset \C^{n+1}\}.
\]
We refer to \cite{Brion,AF:book} for relevant background.  The small quantum cohomology ring of $\mathcal Fl_{n+1}$ was first described in \cite{GiventalKim:TodaFl,C-F:QCohFl}. 
In practical terms, it is a deformed ring structure on $H^*(\mathcal Fl_{n+1},\C)\otimes \C[q_1,\dotsc,q_{n}]$ that has positive (enumerative geometric) structure constants with respect to the basis $\{\sigma^w\mid w\in S_{n+1}\}$ of Schubert classes, given by $3$-point genus $0$ Gromov-Witten invariants. As a ring it is moreover generated by the Chern classes $c_1(V_{k}/V_{k-1})$ of the tautological line bundles, and there is a generalised Borel presentation $\C[x_1,\dotsc, x_{n+1};q_1,\dotsc,q_n]/(E_1^q,\dotsc, E_{n+1}^q)$ involving $q$-deformed elementary symmetric polynomials, where $x_k=-c_1(V_{k}/V_{k-1})$. The degree~$2$ Schubert classes are first Chern classes of the vector bundles $ V^*_{k}$, namely $\sigma^{s_k}=x_1+\dotsc+ x_k$.
We also recall that  \cite{FGP} gives stable polynomial representatives for all the Schubert classes in the quantum cohomology ring that specialise to Schubert polynomials \cite{LasSch:Schubert} for $q_1,\dotsc, q_n=0$.

\subsubsection{Peterson theory}\label{s:QCoh}The finite Toeplitz matrices $\Toeplitz_{n+1}(\C)$ were related to the quantum cohomology rings of partial flag varieties by Dale Peterson \cite{peterson,rietsch,Kos:QCoh}, with the open stratum (where all $\Delta_i\ne 0$) relating to the full flag variety. Namely,
\begin{eqnarray}\label{e:Petersoniso}
qH^*(\mathcal Fl_{n+1},\C)[q_i\inv]&\overset{\sim}{\longrightarrow} &\C[\ToeplitzU_{n+1}(\C)][\Delta_i\inv]
\\
\sigma^{s_k}& \mapsto & \mathfrak S^{s_k}:=\Minor^{[k-1]\cup\{k+1\}}_{[k]+n-k+1}/\Delta_k,\notag
\end{eqnarray}
see \cite[Theorem~37]{Kos:QCoh}, and \eqref{e:mviaminors} for notation. Moreover, the whole quantum Schubert basis $\{\sigma^w\mid w\in S_{n+1}\}$ translates to a basis  of functions $\{\mathfrak S^w\mid w\in S_{n+1}\}$ (over $\C[\Delta_1^{\pm 1},\dotsc, \Delta_n^{\pm 1}]$) that takes  positive values on $\ToeplitzU_{n+1}(\R_{>0})$,
 see \cite{rietschJAMS}, and the quantum parameters themselves are given by the same formula as the components $q_k$ of the parametrisation map $\mathbf q$ from  \eqref{e:q-paramToeplitz} \cite[Corollary~28]{Kos:QCoh}. Though we will not use it here, we remark that Peterson also gave a related isomorphism to the homology of the loop group of $SU(n+1)$ and the above constructions extend to general $G/P$, see for example \cite{Gin:GS,peterson,Kos:QCoh,LamShimozono:2010,YunZhu:Loop,Chow:PLSTheorem,LamRietsch}.
 
 \subsubsection{Mirror symmetry}\label{s:MirrorSymm}
Another perspective on $qH^*(\mathcal Fl_{n+1},\C)[q_1\inv,\dotsc,q_{n}\inv]$ is as Jacobi ring of a function called the `superpotential' $\FFlie$ of the flag variety. This function was  constructed in Lie theoretic terms in \cite{rietsch} extending a Laurent polynomial formula from \cite{Givental:QToda}, and it leads to a  description of Toeplitz matrices 
as critical points. The same function as the Lie-theoretic superpotential was already constructed by Berenstein and Kazhdan in a representation-theoretic context, see \cite{BeKa}, where it arose as the `decoration' of a geometric crystal. There is much extended work including recent proofs of key mirror symmetry predictions for $G/P$, some using the connection with \cite{BeKa}, for example  \cite{Chhaibi:Thesis,Chow:Flag,Chow:FlagDmirror,Chow:Gamma,
LamTemplier,MR,RW}.

The superpotential $\FFlie$ is defined on $Z_{\Lie}=B_+\cap U_- T\bar w_0 U_-$ (working over $\C$), where $U_-$ is the lower-triangular analogue of $U_+$ and $B_+=TU_+$, and where $\bar w_0=((-1)^i\delta_{i,n+2-j})_{i,j}$ represents the longest element $w_0\in S_{n+1}$. Namely,  
\begin{equation}\label{e:FFlie}
\FFlie(b=u_Ld\bar w_0u_R)=\sum_{k=1}^n f_k^*(u_L)+\sum_{k=1}^n f_{k}^*(u_R)
\end{equation}
in terms of the homomorphism $f_k^*:U_-\to(\C,+)$  corresponding to the negative Chevalley generators $f_k$. If $d\in T$ is fixed, then the critical points of  $
\mathcal F^{(d)}_{\Lie}:B_+\cap U_- d w_0 U_-\longrightarrow\C
$ 
are precisely the upper-triangular Toeplitz matrices with fixed minors $\Delta_i$ determined by $d$ as in  \eqref{e:di}, see~\cite{rietsch}. 
This superpotential should be considered as a mirror-symmetric representation of a particular anticanonical divisor of the Langlands dual $\mathcal Fl_{n+1}$, which is made up of $n$ Schubert divisors 
and $n$ opposite Schubert divisors. 
Correspondingly, we have the natural splitting of  $\FFlie$  into $2n$ summands seen in \eqref{e:FFlie}, and  we think of the individual summands as in bijection with  these $2n$ irreducible divisors from above. Indeed, they should have an interpretation as GHKK theta functions  \cite{GHKK} associated to these divisors in the sense of \cite{RW,BossingerCMN24,Johnston}. Concretely, we have for the restrictions of the summands of $\FFlie$ to the Toeplitz part of $Z_{\Lie}$ that they satisfy $\mathfrak S^{s_k}(b)=f_{k}^*(u_L)=f_{n-k+1}^*(u_R)$. Thus  $\FFlie|_{\ToeplitzU_{n+1}}=2\sum_k\mathfrak S^{s_k}$ represents the anticanonical class in $qH^*(\mathcal Fl_{n+1})[q_i\inv]$. This observation in its earliest form is found in \cite[Corollary~2]{Givental:QToda}, and is proved in increasing generality in \cite{LRYZ,Chow:FlagDmirror}.

\subsection{Asymptotics of quantum parameters} \label{s:RealAsymptotic}

We note that it is very difficult to ever write down the associated totally positive Toeplitz matrix $u$ given some parameter values $(q_1,\dotsc, q_n)$. The matrix $u$ is simply shown to exist, either via a Perron-Frobenius eigenvector argument \cite{rietschJAMS}, or as positive critical point of the superpotential  \cite{rietsch}. This is very much in contrast to the Edrei theorem, which readily associates a totally nonnegative Toeplitz matrix to any choice   of Schoenberg parameters. Namely, via the map 
\begin{eqnarray}\label{e:Cpower}
\mathcal E:\Omega_{S}\times\R_{\ge 0} &\longrightarrow &\ToeplitzU_{\infty}(\R_{\ge 0})\\
\quad ((\boldalpha,\boldbeta),\gamma)& \mapsto &\qquad  u(\mathbf c),\notag
\end{eqnarray}
where $u(\mathbf c)$ is given in \eqref{e:SchoenbergFormula}.
To nevertheless relate the two, 
we consider open dense subsets of $\ToeplitzU_\infty(\R_{\ge 0})$, 
\begin{equation}
\ToeplitzU^{\circ\circ}_\infty(\R_{>0})\subset\ToeplitzU^\circ_\infty(\R_{>0})\subset\ToeplitzU_\infty(\R_{>0})\subset \ToeplitzU_\infty(\R_{\ge 0}),
\end{equation}
where 
$\ToeplitzU_\infty(\R_{>0})$ is the totally \textit{positive} subset of $\ToeplitzU_\infty(\R_{\ge 0})$, with all nontrivial minors positive, next  
$\ToeplitzU^\circ_\infty(\R_{>0})$ is  the subset for which all Schoenberg parameters $\alpha_i,\beta_j$ are nonzero, and $\ToeplitzU^{\circ\circ}_\infty(\R_{>0})$ where they are moreover distinct, $\alpha_i>\alpha_{i+1}$ and $\beta_j>\beta_{j+1}$. In other words, $u(\mathbf c)$ lies in $\ToeplitzU^\circ_\infty(\R_{>0})$ if and only if the generating function $1+\sum_{k=1}^\infty c_k x^k$ has infinitely many roots and poles, and  in $\ToeplitzU^{\circ\circ}_\infty(\R_{>0})$ if these roots and poles are additionally all simple.

\begin{theorem}\label{t:IntroVKtyp} Suppose $(u^{(n+1)})_n$  is a sequence of (finite) totally positive Toeplitz matrices  $u^{(n+1)}\in  \ToeplitzU_{n+1}(\R_{>0})$ that converges uniformly to an element $u(\mathbf c)\in\ToeplitzU^\circ_{\infty}(\R_{>0})$, and let 
%Now denote by  the finite the projections of $u(\mathbf c)$, and 
$(d_1^{(n+1)},\dotsc, d^{(n+1)}_{n+1})$ be the positive  parameters associated to $u^{(n+1)}$ in Theorem~\ref{t:IntroRfin}. Then 
  \begin{equation}\label{e:VKtyp}
\lim_{n\to\infty}\sqrt[n] {d_i^{(n+1)}}=\alpha_i\qquad\text{and}\qquad
\lim_{n\to\infty}{\sqrt[n] {d_{n+2-j}^{(n+1)}}}=\beta_j\inv.
\end{equation}
\end{theorem}

Here for $(u^{(n+1)})_{n\in\N}$ to converge uniformly to $u(\mathbf c)$ means that for every $\ep>0$ there is an $N$ such that $u^{(n+1)}$ lies in the $\ep$-neighborhood of the truncation $u^{[n+1]}$ of $u(\mathbf c)$ in $\ToeplitzU_{n+1}(\R)\cong \R^n$ whenever $n>N$.

\begin{remark}\label{r:AsymptoticsReal}
Note that the positive diagonal matrix $\Delta(u^{(n+1)})$ has no particular structure other than its entries  multiplying to $1$. However, after taking the normalised limit coordinate-wise starting from the first diagonal entry $d_1$ it becomes `dominant', in the sense that the limit diagonal entries form a weakly decreasing sequence (namely given by $(\alpha_i)_i$). And if instead we start from the last entry, $d_{n+1}$, then the result is `anti-dominant', $(\beta_j\inv)_j$, consistent with the idea of this `reversal' relating highest to lowest weight in the finite setting. Though, again, in the finite case before taking the limits, there was no such structure visible. 
In terms of the quantum parameters, Theorem~\ref{t:IntroVKtyp}  implies the relationship
\begin{equation}\label{e:limqi}
\lim_{n\to\infty}\sqrt[n]{q_i^{(n+1)}}=\frac{\alpha_{i+1}}{\alpha_i},\qquad
\lim_{n\to\infty}\sqrt[n]{q_{n+1-j}^{(n+1)}}=\frac{\beta_{j+1}}{\beta_{j}}
\end{equation}
with the Schoenberg parameters of $u$. Thus we obtain from the (finite) parameter sequences $(q_1,\dotsc, q_n)$ of the $u^{(n+1)}$ two infinite sequences with limit $0$, namely $\left(\frac{\alpha_2}{\alpha_1},\frac{\alpha_3}{\alpha_2},\dotsc\right)$ and $\left(\frac{\beta_2}{\beta_1},\frac{\beta_3}{\beta_2},\dotsc \right)$.   \end{remark}

Let us now prepare to prove Theorem~\ref{t:IntroVKtyp}. Let $(m_{ij})_{i,j}$ be the standard coordinates for $u\in U^{(\infty)}_+(\R_{>0})$. We have the following convenient formula for the $m_{ij}$ in terms of minors of $u$,
\begin{equation}\label{e:mviaminors}
m_{ij}=\frac{\Minor_{[i]+j}^{[i]}(u)\Minor_{[i-1]+(j-1)}^{[i-1]}(u)}{\Minor^{[i]}_{[i]+(j-1)}(u)\Minor_{[i-1]+j}^{[i-1]}(u)},
\end{equation}
where $[i]+j$ refers to the set $[1+j,i+j]$, and $\Minor^J_K$ is the minor with row set $J$ and column set $I$, \cite[(7.1)]{rietschToeplitz}. These minors if $u$ is Toeplitz appear in the following result of Edrei's, that is part of a full proof of Theorem~\ref{t:EdreiIntro} given in  \cite{Edrei53} as an alternative to  \cite{ASW,Edrei52}.

\begin{prop}[{\cite[Assertion iii]{Edrei53}}]\label{p:Edreilimits} If $u\in\ToeplitzU_\infty(\R_{>0})$, then the (Toeplitz) minors from \eqref{e:mviaminors} converge separately in~$i$ and~$j$ recovering the Schoenberg parameters $\alpha_i$ ($j\to\infty)$ and $\beta_j$ ($i\to\infty$) of $u$. 
\end{prop}
 For example, if $u=u(\mathbf c)$ as in Equation~\ref{e:u}, then the radius of convergence of the generating function $\sum c_ix^i$ is just $\alpha_1\inv$, by Equation~\ref{e:SchoenbergFormula}. Therefore $\alpha_1=\lim_{j\to\infty} \frac{c_{j}}{c_{j-1}}$, which is the first instance of the above proposition. The general statement is then proved recursively, and by using the symmetry that inverts the roles of the $\alpha$'s and $\beta$'s. Related  results for more general patterned matrices have  appeared in \cite{LP:I,LP:II}.
\begin{proof}[{Proof of Theorem~\ref{t:IntroVKtyp}}]
Let $u$ have standard coordinates $(m_{ij})_{i,j\in\N}$. For notational simplicity let us first consider the case where $u^{(n+1)}$ is replaced by the truncation $u^{[n+1]}$  of $u$, so that the standard coordinates are just $(m_{ij})_{i+j\le n+1}$.  We have that 
$\Delta_i(u)=\Minor^{[i]}_{[i]+n-i+1}(u)=
{\prod_{k=1}^{i}\prod_{r=1}^{n-i+1} m_{kr}}$.  Therefore
\begin{equation}\label{e:dn+1}
d^{[n+1]}_{i}=\frac{\Delta_{i}(u)}{\Delta_{i-1}(u)}=\frac{\prod_{k=1}^{i}\prod_{r=1}^{n-i+1}\ m_{kr}}{\prod_{k=1}^{i-1}\prod_{r=1}^{n-i+2} m_{kr}}=(\prod_{r=1}^{n-i+1}\ m_{{i},r})(\prod_{k=1}^{i-1}m_{k,n-i+2})\inv.
\end{equation}
Consider $i$ fixed and let $\mu_{ij}=\ln(m_{ij})$ and $\delta^{[n+1]}_i=\ln(d_i^{[n+1]})$, so that \eqref{e:dn+1} translates to 
\begin{equation}\label{e:deltan+1}
\delta^{[n+1]}_i=\left(\mu_{i1}+\dotsc+\mu_{i,n-i+1}\right)-\left(\mu_{i-1,n-i+2}+\dotsc + \mu_{1,n-i+2}\right),
\end{equation}
with $(n-i+1)$ positive summands, and $(i-1)$ negative summands. Recall that $\alpha_i> 0$. We have that $\lim_{k\to\infty}(\mu_{ik})=\ln(\alpha_i)$ by Proposition~\ref{p:Edreilimits}. 
 Let $\ep>0$ and pick $N_\ep$ such that $|\mu_{hk}-\ln(\alpha_h)|<\ep$ for all $h\in [1,i]$ whenever $k> N_\ep$. Suppose $n$ is large enough so that $n-i+1>N_\ep$. Then we have
\begin{multline*}
|\delta^{[n+1]}_i-n\ln(\alpha_i)|\le 
\\
\begin{array}{l}
\le\left(|\mu_{i1}-\ln(\alpha_i)|+\dotsc+|\mu_{i,n-i+1}-\ln(\alpha_i)|\right)+\left(|\mu_{1,n-i+2}|+\dotsc + |\mu_{i-1,n-i+2}|\right)+(i-1)|\ln(\alpha_i)|\\
\le \sum_{k=1}^{N_\ep}|\mu_{ik}-\ln(\alpha_i)|+ \ep (n-i+1-N_\ep) +\left(\sum _{h=1}^{i-1}(|\mu_{h,n-i+2}-\ln(\alpha_{h})|+|\ln(\alpha_h)|)\right) + (i-1)|\ln(\alpha_i)|\\
\le \sum_{k=1}^{N_\ep}|\mu_{ik}-\ln(\alpha_i)|+ \ep (n-i+1-N_\ep) +\left((i-1)\ep +\sum _{h=1}^{i-1}|\ln(\alpha_h)|\right) + (i-1)|\ln(\alpha_i)|.
\end{array}
\end{multline*}
All of the summands on the right-hand side apart from $\ep(n-i+1-N_\ep)$ are independent of $n$. Thus dividing by $n$ and then taking the limit $n\to\infty$ we find that
\[
\lim_{n\to\infty}\left|\frac{\delta^{[n+1]}_i}n-\ln(\alpha_i)\right|\le \ep.
\]
 We now consider as well as the truncations $u^{[n+1]}$, the more general sequence $u^{(n+1)}$ that is assumed to  converge uniformly to $u$. Consider $\Toeplitz_{n+1}(\R_{>0})$ embedded into $\R^{\mathcal S_{\le n+1}}$ by the logarithmic standard coordinates. Then uniform convergence implies that $|\mu^{(n+1)}_{hk}-\mu_{hk}|<\ep$ for all $h+k\le n+1$ whenever $n>N'_\ep$. It follows that 
\[
\delta^{(n+1)}_i:=\left(\mu^{(n+1)}_{i1}+\dotsc+\mu^{(n+1)}_{i,n-i+1}\right)-\left(\mu^{(n+1)}_{i-1,n-i+2}+\dotsc + \mu^{(n+1)}_{1,n-i+2}\right),
\]
satisfies $|\delta^{(n+1)}_i-\delta^{[n+1]}_i|<n\ep$ for all $i\le n$ whenever $n>N'_\ep$. Therefore
\[
\lim_{n\to\infty}\left|\frac{\delta^{(n+1)}_i}n-\ln(\alpha_i)\right|\le 2\ep.
\]
Since this holds for any $\ep>0$, we obtain $
\lim_{n\to\infty} \frac{\delta^{(n+1)}_i}n=\ln(\alpha_i)
$ and we have proved the first limit formula in \eqref{e:VKtyp}. For the second one, involving $\beta_j$, we observe that the symmetry that sends $\mu_{ij}$ to $\mu_{ji}$ sends $\delta_{i}^{(n+1)} $ to $-\delta_{n-i+2}^{(n+1)}$. This  limit formula then follows using symmetry and that $\lim_{k\to\infty}(\mu_{kj})=\ln(\beta_j)$. 
\end{proof}
\begin{remark}\label{r:infLimQ-paramsEdrei}
 Note that for the infinite Toeplitz matrix with generating function $e^x$, i.e. the infinite Toeplitz matrix $u=u(\mathbf c)$ with entries $c_k=\frac{1}{k!}$, all Schoenberg parameters are equal to zero, and therefore we cannot apply Theorem~\ref{t:IntroVKtyp}. Nevertheless in this case  the formula \eqref{e:VKtyp} still holds, as follows from calculating the  parameters from Theorem~\ref{t:IntroRfin} associated to the truncations $u^{[n+1]}$ of $u$, 
\[\left(d_1^{[n+1]},\dotsc
, d_{n+1}^{[n+1]}\right)=\left(\frac{1}{n!},\frac{1}{(n-1)!},\frac{2}{(n-2)!},\dotsc\dotsc, \frac{(n-2)!}{2},(n-1)!,n!\right ).
\]
Picking up from Remark~\ref{r:AsymptoticsReal} 
let us also consider the normalised limits of the quantum parameters. These turn out to still be well-defined even though all $\alpha_i$ and $\beta_j$ are zero.  Namely, in this case the quantum parameters $\left(q_1^{(n+1)},\dotsc, q_{n}^{(n+1)}\right)$  of the truncation $u^{[n+1]}$ are $(n,(n-1)2,\dotsc, (n-k)k, \dotsc, n)$ and we have $\lim_{n\to\infty}\sqrt[n]{q_{k}^{(n+1)}}=\lim_{n\to\infty}\sqrt[n]{q_{n+1-k}^{(n+1)}}=1$ for all $k$. However, in this case the limit sequences $(1,1,1,\dotsc )$ do not converge to $0$,  in contrast to the $u\in \ToeplitzU^{\circ}_\infty(\R_{>0})$ case from  Remark~\ref{r:AsymptoticsReal}.   
\end{remark}

\subsection{Asymptotics of Schubert classes}\label{s:ChernRoots} 
We recall the explicit Laurent polynomial superpotential of Givental \cite{Givental:QToda} which by \cite{rietsch} can be interpreted as pullback of $\FFlie$ to a torus $\mathbb T_{Giv}$ that embeds into the domain $Z_{\Lie}$. For convenience we  translate Givental's construction to be based on our quiver $Q_{n+1}$, see  Remark~\ref{r:GiventalComparison}, and we will restrict to the totally positive part $\mathbb T_{Giv}(\R_{>0})$ from the outset.

\begin{defn}
Let us consider positive labellings of the arrows $a\in\mathcal A_{n+1}$ of the quiver $Q_{n+1}=(\mathcal V_{n+1},\mathcal A_{n+1})$ 
%from Figure~\ref{f:Qn+1} 
of the following kind,
\[
\mathbb T_{Giv}(\R_{>0}):=\left\{(\ell_a)_{a\in\mathcal A_{n+1}}\left |\ \ell_a\in\R_{>0},\quad \ell_{a_{ij}}\ell_{a'_{ij}}=\ell_{a_{i+1,j}}\ell_{a'_{i,j+1}} \text{ for $i+j< n+1$}
\right.\right\}.
\] 
Next define a function $(q_1,\dotsc, q_n):\mathcal G_{n+1}(\R_{>0})\to \R_{>0}^n$ by $
 q_k((\ell_a)_a):=\ell_{a_{k,n-k+1}}\ell_{a'_{k,n-k+1}}$. After identifying the torus $\mathbb T_{Giv}$ with Givental's one from \cite{Givental:QToda} by using a natural bijection of arrows, the Givental superpotential  for $\mathcal Fl_{n+1}$ is given by taking the sum of the arrow coordinates, so in particular
\begin{eqnarray*}
\mathcal F:\mathbb T_{Giv}(\R_{>0})&\longrightarrow &\ \R_{>0}\\
(\ell_a)_{a\in \mathcal A_{n+1}} &\mapsto & \sum_{a\in\mathcal A_{n+1}}\ell_a.
\end{eqnarray*}
\end{defn}
We have that $\mathcal F$ is the pullback of $\FFlie$ under the embedding $\mathbb T_{Giv}\hookrightarrow Z_{\Lie}$ from \cite{rietsch}. The construction of the embedding is technical and we will not recall it here. 
%The embedding of $\mathbb T_{Giv}$ into $Z_{\Lie}$ is the one defined in \cite{rietsch}, and we don't recall it here. %However, we note that $\mathcal F$ is the pullback of $\FFlie$ to $\mathbb T_{Giv}$ and the embedding is a 
We just note that it is a bijection on positive parts, $\mathbb T_{Giv}(\R_{>0})\cong Z_{\Lie}(\R_{>0})$, and that the decomposition  % of the superpotential $\mathcal F$ 
into the $2n$ summands as in  \eqref{e:FFlie} is as follows. 

\begin{defn}\label{d:Fsummands} For every diagonal vertex  there are two distinguished paths to it from the $0$-vertex, one using only vertical arrows $a_{ij}$, and one using only horizontal arrows $a'_{ij}$. Let $\mathcal F_k$ be the arrow label sum for the vertical  path to $v_{k,n-k+1}$ and  $\mathcal F'_k$ the arrow label sum from the horizontal path to $v_{n-k+1,k}$. We have
\begin{eqnarray}
\mathcal F_k((\ell_{a})_a)&:=&\sum_{i=1}^{k} \ell_{a_{i,n-k+1}}  \qquad \text{corresponding to $f_k^*(u_L)$ in \eqref{e:FFlie}},\\
\mathcal F'_k((\ell_a)_a)&:=&\sum_{j=1}^{k} \ell_{a'_{n-k+1,j}} \qquad \text{corresponding to $f_{k}^*(u_R)$ in \eqref{e:FFlie}.} 
\end{eqnarray}
Every arrow appears in exactly one of these sums, so that $\mathcal F=\sum_{k=1}^n \mathcal F_k+\sum_{k=1}^n\mathcal F'_k$.
\end{defn}   
Now, since we already changed perspective to $Q_{n+1}$, we can easily write down an embedding 
\begin{eqnarray}
\ToeplitzU_{n+1}(\R_{>0})&\hookrightarrow & \mathbb T_{Giv}(\R_{>0})\label{e:ToepIntoGiv} \\
(m_{ij})_{i+j\le n+1} &\mapsto & (\ell_{a})=(v_{h(a)}-v_{t(a)})_{a\in\mathcal A_{n+1}},\notag
\end{eqnarray}
for $(v_{ij})_{i,j}$ as in Definition~\ref{d:Qn+1}. Namely, this is well-defined by Proposition~\ref{p:ToeplitzEquiv}. The composite~embedding,
\begin{equation}\label{e:REmbedding}
\ToeplitzU_{n+1}(\R_{>0})\hookrightarrow Z_{\Lie}(\R_{>0})\overset{{\text{\cite{rietsch}}}}{\longrightarrow} \mathbb T_{Giv}(\R_{>0}),
\end{equation}
which could a priori be different, can be shown to agree with \eqref{e:ToepIntoGiv} using \cite[Theorem~4.3.1]{Ludenbach}. This perspective on \eqref{e:ToepIntoGiv} could  be expanded into an alternative proof of  Proposition~\ref{p:ToeplitzEquiv}. Note our $m_{ij}$ are not coordinates on $Z_{\Lie}$, though, and are distinct from the (inverted) ideal coordinates of \cite{Ludenbach}, see Remark~\ref{r:GiventalComparison}. 
%We finally note that  \eqref{e:ToepIntoGiv} and \eqref{e:REmbedding} are the same embedding by work of L\"udenbach \cite{Ludenbach}. Moreover, conversely, the embedding \eqref{e:REmbedding} and  \cite[Theorem~4.3.1]{Ludenbach}, which can be used to relate our coordinates to (inverse) ideal coordinates of \cite{Ludenbach} in the Toeplitz case, could be turned into an alternative proof of Proposition~\ref{p:ToeplitzEquiv}. See also Remark~\ref{r:GiventalComparison}. 

Let us now write $\mathcal F_k|_{\ToeplitzU_{n+1}}$  for the pullback  of $\mathcal F_k$ under \eqref{e:ToepIntoGiv} and similarly for $\mathcal F'_k$, and note that these pullbacks agree with the restrictions of their corresponding terms in $\FFlie$. 
\begin{lemma}\label{l:Fks} The superpotential $\FFlie$ restricted to $\ToeplitzU_{n+1}(\R_{>0})$ is given in standard coordinates  by 
\[
\FFlie|_{\ToeplitzU_{n+1}}((m_{ij})_{i,j})=2\sum_{k=1}^n v_{k,n-k+1},
\] 
using the diagonal vertex labels, see \eqref{e:vij}. Moreover for the individual summands of $\FFlie$
we have
\begin{eqnarray}
\mathcal F_k|_{\ToeplitzU_{n+1}}((m_{ij})_{i,j}) &= &v_{k,n-k+1}=\frac {1}{m_{k,n-k+1}}+\sum_{r=1}^{\min(k-1,n-k)}\frac {1}{m_{k-r,n-k+1-r}},
\label{e:Fk}\\
\mathcal F'_k|_{\ToeplitzU_{n+1}}((m_{ij})_{i,j})&= &v_{n-k+1,k}=
\frac {1}{m_{n-k+1,k}}+\sum_{r=1}^{\min(k-1,n-k)}\frac {1}{m_{n-k+1-r,k-r}}.
\label{e:Fk'}
\end{eqnarray}
\end{lemma}
Note that the above formulas only make sense after restriction, as our $m_{ij}$ are not coordinates on $Z_{\Lie}$. % We combine the proof with the proof of the following corollary. 
\begin{cor}\label{c:Fasymptotics}
Suppose the sequence $(u^{(n+1)})_n$ with $u^{(n+1)}\in \ToeplitzU_{n+1}(\R_{>0})$ 
%given in terms of standard coordinates $(m_{ij}^{(n+1)})_{i,j}$ 
converges uniformly to an element $u(\mathbf c)\in \ToeplitzU^\circ_{\infty}(\R_{>0})$ with Schoenberg parameters $\boldalpha=(\alpha_i)_i$ and $\boldbeta=( \beta_j)_j$. Then 
\begin{eqnarray}
\lim_{n\to\infty} \mathcal F_k|_{\ToeplitzU_{n+1}}(u^{(n+1)})&=&\frac{1}{\alpha_1}+\frac{1}{\alpha_2}+\dotsc+\frac{1}{\alpha_k},
\label{e:limFk}\\
\lim_{n\to\infty} \mathcal F'_k|_{\ToeplitzU_{n+1}}(u^{(n+1)})&=&\frac{1}{\beta_1}+\frac{1}{\beta_2}+\dotsc+\frac{1}{\beta_k}
.\label{e:limFk'}
\end{eqnarray}
The value of the full superpotential $\FFlie|_{\ToeplitzU_{n+1}}(u^{(n+1)})$ diverges as $n\to\infty$. 
\end{cor}

\begin{proof}[Proof of Lemma~\ref{l:Fks} and Corollary~\ref{c:Fasymptotics}]
The definition of $\mathcal F_k$ and $\mathcal F_k'$ combined with the embedding \eqref{e:ToepIntoGiv} leads to a telescopic sum that immediately implies the formulas \eqref{e:Fk} and \eqref{e:Fk'}.   Note that the number of summands in $\mathcal F_k$ and $\mathcal F'_{k}$ is simply $k$ (for large $n$). The limit formulas  \eqref{e:limFk} and \eqref{e:limFk'} now follow from Proposition~\ref{p:Edreilimits} together with an analogous argument to the one in Theorem~\ref{t:IntroVKtyp} concerning the convergence. The value of $\FFlie$ diverges since the summands grow and their number goes to infinity. %we have that $\mathcal F_{k} 
\end{proof}

\begin{remark}[Chern roots]\label{r:ChernRoots}
The two sides of the identity $\mathcal F_k|_{\ToeplitzU_{n+1}}=\mathcal F'_{n-k+1}|_{\ToeplitzU_{n+1}}$, see  Lemma~\ref{l:Fks}, can be interpreted as relating under \eqref{e:Petersoniso} to the following two different expressions for $\sigma^{s_k}$ in $qH^*(\mathcal Fl_{n+1})[q_i\inv]$,
\begin{equation*}\label{identities}
\begin{array}{ccccccc}
c_1(V_k^*)&=&x_1+\dotsc+x_k& =&c_1(\C^{n+1}/V_{k})&=&-x_{n+1}-x_n-\dotsc -x_{k+1}.
\end{array}
\end{equation*}
Thus Corollary~\ref{c:Fasymptotics} describes the functions $\sum_{i=1}^k 1/\alpha_i$ and $\sum_{j=1}^{k} 1/\beta_j$ on $\ToeplitzU_\infty^\circ(\R_{>0})$ in terms of asymptotic limits of the first Chern classes $c_1(V_{k}^*)$  and $c_1(\C^{n+1}/V_{n-k+1})$, respectively, after applying the Peterson isomorphism~\eqref{e:Petersoniso} and restricting to $\ToeplitzU_{n+1}(\R_{>0})$.  
Moreover, $\frac{1}{\alpha_1},\dotsc, \frac{1}{\alpha_k}$ and $\frac 1{\beta_1},\dotsc,\frac 1{\beta_k}$ individually can be thought of as asymptotic limits of Chern roots of these rank $k$ vector bundles (as $n\to\infty$). Restated another way, the roots and poles of the generating function~\eqref{e:SchoenbergFormula} in Edrei's theorem gain natural interpretations as asymptotic limits of the generators of $qH^*(\mathcal Fl_{n+1},\C)$. Namely, the $i$\textsuperscript{th} pole $1/\alpha_i$ is the  asymptotic limit of $x_i=c_1((V_i/V_{i-1})^*)$, and the $j$\textsuperscript{th} root $-1/\beta_j$ the asymptotic limit of $x_{n-j+2}=c_1((V_{n-j+2}/V_{n-j+1})^*)$. If we formally introduce a normalised Chern polynomial $\bar c_t(V)=c_t(V)/e(V)$, where $e(V)$ is the Euler class of $V$, for example $\bar c_t(V_k)={\prod_{i=1}^k(t-x_i\inv)}$,  then partial products of \eqref{e:SchoenbergFormula} emerge as limits after substitution of $t=x\inv$ into
\begin{eqnarray}\label{e:Chernpol}
\frac{\bar c_t(\C^{n+1}/V_{n-k+1})}{\bar c_t(V_k)}
&=&  \frac{\prod_{j=1}^k(1-x_{n-j+2}\inv t\inv)}{\prod_{i=1}^k(1-x_i\inv t\inv)}\quad \longrightarrow\quad \frac{\prod_{j=1}^k(1+\beta_{j} t\inv)}{\prod_{i=1}^k(1-\alpha_i t\inv)} \qquad (n\to\infty).
\end{eqnarray} 
\end{remark}

Corollary~\ref{c:Fasymptotics} and Remark~\ref{r:ChernRoots} now allow us to describe the asymptotic limits of a general Schubert class $\sigma^w$ in $qH^*(\mathcal Fl_{n+1},\C)$. 
Fix $w\in S_{\infty}$, i.e. $w\in S_k$ for some $k$, and also
 consider the related permutation  $w^{(n+1)}$, 
which is obtained by conjugating by the longest element of $S_{n+1}$. For example, if $w=s_k$ is the $k$-th simple transposition, then $w^{(n+1)}=s_{n-k+1}$ if $n\ge k$ (and $s_k$ otherwise).
\begin{theorem}\label{t:Schubertasymptotics}
Suppose the sequence $(u^{(n+1)})_n$ with $u^{(n+1)}\in \ToeplitzU_{n+1}(\R_{>0})$  converges uniformly to an element $u(\mathbf c)\in \ToeplitzU^{\circ\circ}_{\infty}(\R_{>0})$ that has Schoenberg parameters $\boldalpha=(\alpha_i)_i$ and $\boldbeta=( \beta_j)_j$. Let $w\in S_\infty$ and associate $\mathfrak S^{w}, \mathfrak S^{w^{(n+1)}}\in \C[\ToeplitzU_{n+1}(\C)][\Delta_i\inv]$ for $n$ large enough, using \eqref{e:Petersoniso}. Then
\begin{eqnarray}
\lim_{n\to\infty}\mathfrak S^w(u^{(n+1)})&=&\mathcal  S_w\left(\frac{1}{\alpha_1},\frac{1}{\alpha_2},\frac{1}{\alpha_3},\dotsc\right),\label{e:limSw}\\
\lim_{n\to\infty}\mathfrak S^{w^{(n+1)}}(u^{(n+1)})&=&\mathcal S_w\left(\frac{1}{\beta_1},\frac{1}{\beta_2},\frac{1}{\beta_3},\dotsc\right),\label{e:limSwn}
\end{eqnarray}
where $\mathcal S_w$ is the classical Schubert polynomial 
associated to $w$.     
\end{theorem}
In the case of a simple reflection  we have $ \mathcal S_{s_k}(x_1,x_2,\dotsc)=x_1+\dotsc +x_k$, while $\mathfrak S^{s_k}=\mathcal F_k|_{\ToeplitzU_{n+1}}$ and $\mathfrak S^{s_{n-k+1}}=\mathcal F'_k|_{\ToeplitzU_{n+1}}$, so that for  $w=s_k$ Theorem~\ref{t:Schubertasymptotics} follows from Corollary~\ref{c:Fasymptotics}.
\begin{proof}
We use the quantum Schubert polynomial $\mathcal  S^q_{w}(x_1,x_2,\dotsc;q_{1},q_2,\dotsc)$ of \cite{FGP}. Note that this is a well-defined polynomial associated to $w$ independently of $n$, by the stability property from \cite[Theorem~10.1]{FGP}. Let us always choose $n$ large enough so  that $\mathcal  S^q_{w}=S^q_w(x_1,x_2,\dotsc, x_{n+1};q_{1},\dotsc, q_n)$, and let $q_i^{(n+1)}$ and $x_i^{(n+1)}$ be the values of $q_i$ and $x_i$ on $u^{(n+1)}$ using the Peterson isomorphism \eqref{e:Petersoniso}. Then 
\begin{equation}\label{e:qSchubert}
\mathfrak S^w(u^{(n+1)})=\mathcal S^q_{w}\left(x_1^{(n+1)},x_2^{(n+1)},\dotsc, x_{n+1}^{(n+1)}; q_1^{(n+1)},q_2^{(n+1)},\dotsc, q_{n}^{(n+1)} \right).
\end{equation} Fix a $k$. By assumption on $u(\mathbf c)$ we have $ \alpha_k>\alpha_{k+1}$, and therefore $\lim_{n\to\infty}\sqrt[n]{q^{(n+1)}_k}=\frac{\alpha_{k+1}}{\alpha_k}<1$, see \eqref{e:limqi}. Consequently we must have $\lim_{n\to\infty}{q^{(n+1)}_k}=0$. Similarly, $\beta_{k}>\beta_{k+1}$ implies that $\lim_{n\to\infty}{q^{(n+1)}_{n-k+1}}=0$. 
Moreover, Corollary~\ref{c:Fasymptotics} implies that $x^{(n+1)}_k$ converges to $\frac 1 {\alpha_k}$ and $-x^{(n+1)}_{n-k+1}$ to $\frac 1 {\beta_k}$, see Remark~\ref{r:ChernRoots}. Let us now take the limit $n\to\infty$ on both sides of \eqref{e:qSchubert}. The $x_k^{(n+1)}$ converge to $\frac 1{\alpha_k}$ and the $q_k^{(n+1)}$ to $0$, and since $\mathcal S^q_w$ is a polynomial (independent of $n$), we can take the coordinate-wise limit and \eqref{e:limSw} follows. 
To prove the remaining identity \eqref{e:limSwn}, we consider the involution on $qH^*(\mathcal Fl_{n+1}\C)$ that takes  $x_i\leftrightarrow -x_{n-i+1}$ and $q_k\leftrightarrow q_{n-k+1}$. This involution swaps $\sigma^{w}\leftrightarrow \sigma^{w^{(n+1)}}$ and applied on both sides of \eqref{e:qSchubert} via \eqref{e:Petersoniso} gives
\[
\mathfrak S^{w^{(n+1)}}(u^{(n+1)})=\mathcal S^q_{w}\left(-x^{(n+1)}_{n+1},\dotsc, -x_{1}^{(n+1)};q^{(n+1)}_n,\dotsc, q^{(n+1)}_1\right).
\]
Now we can again take the limit $n\to\infty$ and the identity \eqref{e:limSwn} follows.   %\cite[p 214]{AF:book}
\end{proof}
\subsection{The multiplication table of $\N$}\label{s:MultiplicationTable} 
Before carrying on to the next  asymptotic result, which involves a tropicalisation of the finite and infinite Toeplitz matrices, we prove the following elementary  statement about divergence-free labellings of $Q_\infty$. This is an application (and equivalent restatement) of Edrei's theorem, that the only totally positive sequences whose generating function is entire with no zeros are of the form  $e^{\gamma x}$. Let us call a vertex labelling of the infinite quiver $Q_\infty$ from Figure~\ref{f:QinfQ5} \textit{normalised} if the vertex $(1,1)$ is labelled by $1$. Normalised divergence-free positive vertex labellings of $Q_\infty$ correspond to characters of the infinite symmetric group, by Proposition~\ref{p:ToeplitzEquiv} combined with Thoma's Theorem~\ref{t:ThomaIntro}. 

\begin{theorem}\label{t:MultiplicationTable}
The vertex labelling of $Q_\infty$ given by the multiplication table of the natural numbers,
\[ (i,j)\mapsto v_{ij}:=i\cdot j\ 
\] 
is divergence-free, and it is the unique normalised divergence-free labelling $(v_{ij})_{i,j}$ for which the vertex labels $v_{ij}$ are unbounded along the first row and column.  
\end{theorem} 
This theorem implies also that the multiplication table is  the only \textit{integral} positive, normalised, divergence-free labelling. Note that this very special  positive vertex labelling of $Q_\infty$ is precisely the labelling associated to the Toeplitz matrix with generating function $e^x$, or to the regular representation  in Thoma's theory.

\begin{proof}[Proof of Theorem~\ref{t:MultiplicationTable}]
It is straightforward that the multiplication table is a divergence-free labelling of $Q_\infty$. Suppose conversely that $(v_{ij})_{i,j}$ is a divergence-free positive vertex labelling  that is unbounded along the first row and column. Then $(m_{ij})_{i,j}$ defined by $m_{ij}=(v_{ij}-v_{i-1,j-1})\inv$ are the standard coordinates for an infinite totally positive  Toeplitz matrix $u$ by Proposition~\ref{p:ToeplitzEquiv}. Consider the parameters $((\boldalpha,\boldbeta),\gamma)$ given by Edrei's  Theorem~\ref{t:EdreiIntro}. If $\alpha_1>0$ then $\lim_{k\to\infty}v_{1k}=\lim_{k\to\infty} m_{1k}\inv=\alpha_1\inv$   which contradicts the assumption that  $\{v_{1j}\}$ is bounded. Therefore $\alpha_1=0$. Similarly for $\beta_1$. Therefore all $\alpha_i$ and $\beta_j$  are $0$. It follows that the generating function associated to $u$ is $e^{\gamma x}$. The normalised condition implies $\gamma=1$.   
\end{proof}

\section{Tropical Toeplitz matrices}\label{s:Tropical}

In the previous section we recovered Schoenberg parameters asymptotically from parameters associated to finite totally positive Toeplitz matrices using an approach that involves describing the (finite and infinite) Toeplitz matrices uniformly  in terms of natural coordinates on $U_+(\R_{>0})$. However, once we have chosen this positive chart (or perhaps a class of positive charts related by positive birational transformations), then the next natural thing to do, following \cite{Lusztig94}, is to tropicalise.

\subsection{The tropical analogue of $U_+$ and Lusztig's weight map} To tropicalise $U_+$ and $T_{SL_{n+1}}$ we first  replace  the semifield $\R_{>0}$ in $U_+(\R_{>0})$ and $T_{SL_{n+1}}(\R_{>0})$ by a semifield  with valuation $(\RR_{>0},\Val)$, where 
\begin{center}
\begin{itemize}
\item $\RR_{>0}$ is contained in a ring $\RR$,
\item $\Rmin=(\R,+,\min)$ is the `tropical semifield',
\item  $\Val:\RR_{>0}\to \Rmin$ is a semifield homomorphism. 
\end{itemize} 
\end{center}
The $\Rmin$-tropicalisation of $U_+$ is then constructed from $U_+(\RR_{>0})$ in analogy with \cite{Lusztig94} by applying the equivalence relation $\sim$ given by equality of valuations of  coordinates in a positive chart. If we also assume that $\Val:\RR_{>0}\to\Rmin$ is surjective then the result is independent of the choice of $\RR_{>0}$. In this case we have natural isomorphisms 
$T_{SL_{n+1}}(\Rmin)\cong\mathfrak h_{SL_{n+1},\R}$ and, using the standard coordinate chart,
\begin{eqnarray*}
\Val: U_+(\Rmin)=U_+(\RR_{>0})/\sim & \overset{\sim}\longrightarrow & (\Rmin)^{\mathcal S}.
\end{eqnarray*}
Here $\mathcal S=\mathcal S_{\le n+1}=\{(i,j)\in \Z_{>0}^2\mid i+j\le n+1\}$ in the rank $n$ case, or we may perform the analogous construction using standard coordinates for $U_+^{(\infty)}$, in which case $\mathcal S=\N\times\N$.  
We write $(M_{ij})_{i,j\in\mathcal S}$ for the tropical coordinates of $[u]$ given by $M_{ij}=\Val(m_{ij})$. Apart from the standard chart, there are other charts obtained multiplying the simple root subgroups together along different reduced expressions of $w_0$, see \cite{Lusztig94}, whose tropical coordinates are related by piecewise linear maps to the $((M_{ij})_{i,j})$. Moreover, there is a natural, well-defined subset $U_+(\Z^{>0}_{\min{}})$ of $U_+(\Rmin)$ which parametrises Lusztig's canonical basis \cite{Lusztig:cbJAMS,Lusztig94}, and there is an associated `weight map' to  the root lattice, $U_+(\Z^{>0}_{\min{}})\to \mathfrak h^*_{PSL_{n+1},\Z}$. Let $\alpha_i=\ep_i-\ep_{i+1}$ be the $i$\textsuperscript{th} simple~root. We now consider the weight map over $\R$ and express it in terms of standard coordinates as a shortcut to defining it. A description closer to the original definition is given in Remark~\ref{r:L}. 

\begin{defn}[{\cite[2.8]{Lusztig:cbJAMS}}] The Lusztig weight map ($\otimes \R$) is the map given in standard coordinates by
\begin{eqnarray*}\label{e:L}
\mathcal L: U_+(\Rmin)&\longrightarrow &\mathfrak h^*_{PSL_{n+1},\R}\\
(M_{ij})_{i+j\le n+1} &\mapsto & \sum_{k=1}^{n+1}\left(\sum_{j=1}^{n+1-k} M_{kj}- \sum_{i=1}^kM_{i,n+1-k}\right)\ep_k.
\end{eqnarray*}
 Note that the right-hand side can be rewritten as a linear combination of the simple roots $\alpha_i=\ep_i-\ep_{i+1}$.
\end{defn}

\begin{remark}\label{r:L} We can graphically represent this map as follows. Arrange the $M_{ij}$  into a lower-triangular tableau, analogously to the $m_{ij}$ in Section~\ref{s:StandardCoord}, to represent a tropical point $[u]\in U_+(\Rmin)$. 
Then associate to each box a positive root $\alpha_{k,\ell}=\ep_k-\ep_\ell$ as shown on the right-hand side in the example $n=3$ below. The $\ep_1,\dotsc,\ep_{4}$ along the diagonal have just been added in for illustration. 
\begin{center}
 \begin{tikzpicture}[scale=0.75]
\node[fill=white] at    (0.5,2.54) {$M_{31}$}  ;
\node[fill=white] at    (0.5,1.54) {$M_{21}$}  ;
\node[fill=white] at    (0.5,0.55) {$M_{11}$}  ;
\node[fill=white] at    (1.8,1.54) {$M_{22}$}  ;\node[fill=white] at    (4.5,1.54) {$=$} ; 
\node[fill=white] at    (1.8,0.55) {$M_{12}$}  ;
\node[fill=white] at    (3,0.55) {$M_{13}$}  ;
\draw[black] (-.2,0.1) -- (-.2,3.1) -- (1.15,3.1) -- (1.15,0.1) -- (3.6,0.1) -- (3.6,1.1)--(-.2,1.1)--(-.2,0.1)--(2.4,0.1)--(2.4,2.1)--(-.2,2.1);
 \end{tikzpicture}\quad
  \begin{tikzpicture}[scale=0.75]
\node[fill=white] at    (0.5,2.54) {$M_{\alpha_{34}}$}  ;
\node[fill=white] at    (0.5,1.54) {$M_{\alpha_{24}}$}  ;
\node[fill=white] at    (0.5,0.55) {$M_{\alpha_{14}}$}  ;
\node[fill=white] at    (1.8,1.54) {$M_{\alpha_{23}}$}  ;
\node[fill=white] at    (1.8,0.55) {$M_{\alpha_{13}}$}  ;
\node[fill=white] at    (3,0.55) {$M_{\alpha_{12}}$}  ;
\draw[black] (-.2,0.1) -- (-.2,3.1) -- (1.15,3.1) -- (1.15,0.1) -- (3.6,0.1) -- (3.6,1.1)--(-.2,1.1)--(-.2,0.1)--(2.4,0.1)--(2.4,2.1)--(-.2,2.1);
 \end{tikzpicture},\qquad\quad \qquad\qquad
\begin{tikzpicture}[scale=0.7]
\node[fill=white] at    (0.5,3.55) {$\ep_4$};
\node[fill=white] at    (0.5,2.55) {$\alpha_{34}$}  ;
\node[fill=white] at    (0.5,1.55) {$\alpha_{24}$}  ;
\node[fill=white] at    (0.5,0.55) {$\alpha_{14}$}  ;
\node[fill=white] at    (1.8,2.55) {$\ep_3$};
c\node[fill=white] at    (1.8,1.55) {$\alpha_{23}$}  ;\node[fill=white] at    (1.8,0.55) {$\alpha_{13}$}  ;
\node[fill=white] at    (3.0,1.55) {$\ep_2$};
\node[fill=white] at    (3,0.55) {$\alpha_{12}$}  ;
\node[fill=white] at    (4.2,0.55) {$\ep_1$};
\draw[black] (-.2,0.1) -- (-.2,3.1) -- (1.15,3.1) -- (1.15,0.1) -- (3.6,0.1) -- (3.6,1.1)--(-.2,1.1)--(-.2,0.1)--(2.4,0.1)--(2.4,2.1)--(-.2,2.1);
 \end{tikzpicture},
\end{center}
If we relabel the $M_{ij}$ according to their associated positive root as shown, then the Lusztig weight map is just the map sending $(M_\alpha)_{\alpha\in R_+}$ to the  linear combination $\sum_{\alpha\in R_+}M_\alpha\alpha$ of positive roots. 
Note that our convention is to use a lower-triangular array, but if we rotate by $180^\circ$ this allocation of positive roots to boxes becomes the standard one. We now recall the following lemma from \cite{rietschToeplitz}.% that detropicalises $\mathcal L$.
\end{remark}

\begin{lemma}[{\cite[Lemma~4.6]{rietschToeplitz}}]\label{l:dunilemma} Consider the positive map  $\mathbf d:U_+\cap B_-\dot w_0 B_-\to T_{SL_{n+1}}$ sending $u$ to the diagonal matrix with entries
\begin{equation}\label{e:di-again}
d_1=\Delta_{1}(u),\   
d_2=\frac{\Delta_{2}(u)}{\Delta_{1}(u)},\ 
\dotsc,%\ d_n=\frac{\Delta_{n}(u)}{\Delta_{{n-1}}(u)},
\ d_{n+1}= \frac{1}{\Delta_{n}(u)}.
\end{equation}
Its tropicalisation $\Trop(\mathbf d):U_+(\R_{\min})\to\mathfrak h_{SL_{n+1},\R}$ is precisely Lusztig's weight map $\mathcal L$, after canonically identifying $\mathfrak h_{SL_{n+1},\R}$ with $\mathfrak h_{PSL_{n+1},\R}^*$. 
\end{lemma}

Our first goal will be to give the tropical analogue of the parametrisation theorem Theorem~\ref{t:IntroRfin}. Therefore  we note that the parametrisation map from Theorem~\ref{t:IntroRfin} is precisely  the restriction $\mathbf d|_{\ToeplitzU_{n+1}(\R_{>0})}$ of the map $\mathbf d$ from Lemma~\ref{l:dunilemma}, and its tropical analogue will therefore be the restriction of the Lusztig weight map $\mathcal L$ to the tropical analogue of $\ToeplitzU_{n+1}(\R_{>0})$ -- which is to be determined next.

\subsection{Tropical Toeplitz matrices}
Recall that we represent an element of $U_+(\Rmin)$ by its standard coordinates, which we furthermore arrange in a triangular array. We make the following definition.
\begin{defn}\label{d:minIdeal} Let us call $(M_{ij})_{i+j\le n+1}\in U_+(\Rmin)$ a (rank $n$) min-ideal filling if for every configuration
\begin{center}
\begin{tikzpicture}[scale=0.95]
\node[fill=white] at    (0.52,2.6) {$M_{i+1,j}$}  ;
\node[fill=white] at    (0.5,1.55) {$M_{ij}$}  ;
\node[fill=white] at    (1.74,1.55) {$M_{i,j+1}$}  ;
\draw[black] (-.1,1.05)--(-.1,3.1) -- (1.15,3.1) -- (1.15,1.05)-- (2.3,1.05)--(-.1,1.05)--(-.1,2.1)--(2.3,2.1)--(2.3,1.05); 
\node[fill=white] at (6.6,2) {of standard coordinates we have the relation,};
\node[fill=white] at (12.9,1.995) {$\min(M_{i+1,j},M_{i,j+1})=M_{ij}$.};
 \end{tikzpicture}
 \end{center}
 We write $\Imin_{n+1}$ for the set of rank $n$ min-ideal fillings. 
 We also define two infinite analogues 
\begin{eqnarray*}
 \Imin_{\infty}(\Rmin)&=&\{(M_{ij})_{i,j\in\N}\in \R^{\N\times \N}\mid \min(M_{i+1,j},M_{i,j+1})=M_{ij} \text{ for all $i,j\in\N$} \}\\
 \Imin^s_{\infty}(\Rmin)&=&\{(M_{ij})_{i,j\in\N}\in \Imin_{\infty}\mid \lim_{k\to\infty}M_{ik}=\max_k\{M_{ik}\}, \lim_{k\to\infty}M_{kj}=\max_k\{M_{kj}\}\}.
 \end{eqnarray*}
We call $\Imin_{\infty}(\Rmin)$ the set of infinite min-ideal fillings, and $\Imin_{\infty}^s(\Rmin)$ the set of \textit{stable} min-ideal fillings. 
\end{defn}
The above definition of $\Imin_{n+1}$ is inspired by the definition of an ideal filling given in \cite{Judd:Flag}. Two differences are that we have replaced the $\max{}$-condition of Judd by a $\min{}$-condition, and we do not make any assumption on the signs of the coordinates. A third difference is the interpretation of the min-ideal fillings as elements of $U_+(\Rmin)$, which comes from  recontextualising the work of L\"udenbach \cite{Ludenbach}.

\begin{remark}
With the alternative indexing of the coordinates $(M_{ij})$ by positive roots, $(M_{\alpha})_{\alpha\in R_+}$, the min-ideal condition can be equivalently restated as 
\begin{equation}\label{e:MinIdealViaR+}
M_{\alpha+\beta}=\min(M_{\alpha},M_{\beta})\quad\text{ whenever } \alpha+\beta\in R_+.
\end{equation}
From this perspective min-ideal fillings turn out to be  canonical elements in $U_+(\Rmin)$ in the sense that their coordinates do not change under the piecewise-linear transformations relating the different tropical Lusztig charts. See \cite[Proposition~5.5.6]{Ludenbach}. Moreover,  in \cite[1.2]{rietschToeplitz} the definition of min-ideal filling is generalised to arbitrary type via \eqref{e:MinIdealViaR+} and shown to give a canonical subset of $U_+(\Rmin)$. 
\end{remark}
We now have the following tropical analogue of the finite parametrisation theorem, Theorem~\ref{t:IntroRfin}. Let us consider the tropicalisation of $U_+(\R_{>0})$ using $\RR_{>0}=\KK_{>0}$, the  semifield of generalised Puiseaux series with positive leading term, see  \cite{Markwig:genPuiseux}. The theorem below is \cite[Theorem~B]{rietschToeplitz}, which combines results from \cite{Judd:Flag,Ludenbach}, as well as using \cite{JuddRietsch:24}. 
\begin{theorem}
\label{t:IntroTropFinite}
Let $\ToeplitzU_{n+1,\KK_{>0}}(\Rmin):=\ToeplitzU_{n+1}(\KK_{>0})/\sim $. Then we have that  $\ToeplitzU_{n+1,\KK_{>0}}(\Rmin)$ is  precisely the set $\Imin_{n+1}$ of min-ideal fillings, and moreover the restriction of Lusztig's weight map \eqref{e:L},  
\[
\mathcal L|_{\ToeplitzU}:\ToeplitzU_{n+1,\KK_{>0}}(\Rmin)\longrightarrow \mathfrak h^*_{PSL_{n+1},\R},
\]
is a bijection. \end{theorem}
\begin{remark} While the map $\mathcal L$ is straightforward to compute, the inverse $R:=(\mathcal L|_{\ToeplitzU})\inv$ is piecewise linear and very complicated. Interestingly, however, for any $\lambda\in \mathfrak h^*_{PSL_{n+1},\R}$ the inverse $R(\lambda)$ expresses $\lambda$ in a canonical way as linear combination of positive roots, see Remark~\ref{r:L}. For example, for $\lambda=2\rho$ this recovers the famous identity $2\rho=\sum_{\alpha\in R_+}\alpha$. This part of the theorem in an equivalent form is already essentially found in \cite{Judd:Flag}, though Judd is focused on dominant weights. 

We conjecture that for any arbitrary complex semisimple  algebraic group $G$, any $\lambda\in\mathfrak h^\star_{G,\R}$ has a unique representation $\lambda=\sum_{\alpha\in R_+}M_\alpha \alpha$ determined by the generalised min-ideal condition \eqref{e:MinIdealViaR+} holding for the coefficients $(M_\alpha)_\alpha\in\R^{R_+}$. See also the  related  conjecture from \cite[1.2]{rietschToeplitz}.
\end{remark}
\subsection{Generalising $\ToeplitzU_\infty(\R_{>0})$ to other semifields}  We may tropicalise $U_+$ using a variety of different positive semifields with valuation $(\RR_{>0},\Val)$, and in this case the tropical analogues of Toeplitz matrices will always give a subset of min-ideal fillings \cite[Proposition~4.18]{rietschToeplitz}. That is,
\begin{equation}\label{e:TropInIdealFillingsR}
% https://q.uiver.app/#q=WzAsMixbMCwwLCJcXFRvZXBsaXR6VV97bisxLFxcUlJfez4wfX0oXFxSbWluKTo9XFxUb2VwbGl0elVfe24rMX0oXFxSUl97PjB9KS9cXHNpbSAiXSxbMiwwLCJcXEltaW5fe24rMX0iXSxbMCwxLCIiLDAseyJzdHlsZSI6eyJ0YWlsIjp7Im5hbWUiOiJob29rIiwic2lkZSI6InRvcCJ9fX1dXQ==
\begin{tikzcd}
	{\ToeplitzU_{n+1,\RR_{>0}}(\Rmin):=\ToeplitzU_{n+1}(\RR_{>0})/\sim \ } & {\ \Imin_{n+1},}
	\arrow[hook, from=1-1, to=1-2]
\end{tikzcd}
\end{equation}
where the inclusion is via valuations of standard coordinates. The question of whether we have equality is one of being able to lift tropical points $(M_{ij})_{i,j}$ to $\RR_{>0}$-totally positive points $(m_{ij})_{i,j}$. When passing to infinite Toeplitz matrices this task of lifting and even our definition of infinite totally positive Toeplitz matrix will start to involve a choice of topology.

From now on, let us extend our semifield with valuation $(\RR_{>0},\Val)$ to a semiring $\RR_{\ge 0}=\RR_{>0}\cup \{0\}$ and with it extend the valuation (semiring homomorphism),
\[
\Val:\RR_{\ge 0}\to\Rmininf=(\R\cup\{\infty\},+\min).
\] 
We also assume that $\RR_{\ge 0}$ has a topology, both compatible with the semiring structure and with $\Val$. Here
the topology of $\Rmininf$ is the usual Euclidean topology of $\R$ with added neighbourhoods $(a,\infty]$ of~$\infty$. 
We can now introduce our preferred infinite generalisation of $\ToeplitzU_{n+1}(\RR_{>0})$ from \cite{rietschToeplitz}. 
\begin{defn}\label{d:Toeprest}  Recall that $U_+^{(\infty)}(\RR_{>0})$ is identified with $\RR_{>0}^{\N\times\N}$ via the standard coordinates $(m_{ij})_{i,j}$. Define the subset of `restricted' (infinite) totally positive Toeplitz matrices by
\begin{equation}\label{e:RestrToep}
\ToeplitzU_\infty^{res}(\RR_{>0}):=\{(m_{ij})_{i,j}\in \ToeplitzU_{\infty}(\RR_{>0})\mid \text{all $\lim_{k\to\infty}m_{ik}$ and $\lim_{k\to\infty}m_{kj}$ exist in $\RR_{\ge 0}$}  \}.
\end{equation}
Applying $\Val$ or its associated equivalence relation, we get our tropicalisation, namely 
\[
\ToeplitzU_{\infty,\RR_{>0}}^{res}(\Rmin):=
\ToeplitzU_\infty^{res}(\RR_{>0})/\sim.
\]
Note that in terms of tropical standard coordinates $\ToeplitzU_{\infty,\RR_{>0}}^{res}(\Rmin)$ is a subset  of $\Imin_\infty(\Rmin)$, as follows from the  embedding   \eqref{e:TropInIdealFillingsR} in the finite case. 
\end{defn}  

\begin{example}\label{ex:usualR}
If we allow the trivial valuation with image $\{0\}$, then the usual totally positive infinite Toeplitz matrices fit into this setup. Namely, in this case we can take $(\RR_{> 0},\Val)$ to be   $(\R_{>0},\Val_0)$ where $\Val_0(r)=0$ for all $r\in\R_{\ge 0}$. This is trivially continuous for the usual topology on $\R_{\ge 0}$, and we obtain 
\[
\ToeplitzU_\infty^{res}(\R_{>0,\Val_0})= \ToeplitzU_\infty(\R_{>0})
\]
by an application of Proposition~\ref{p:Edreilimits}. The tropicalisation in this case just consists of the single constant min-ideal filling where all $M_{ij}=0$.
\end{example}
By the above example, we may think of $\ToeplitzU_\infty^{res}(\RR_{>0})$ as a natural generalisation of the real totally positive Toeplitz matrices, $\ToeplitzU_\infty(\R_{>0})$.
\begin{example}\label{ex:ToepRinf} 
If $\RR_{>0}=\R_{>0}$ but instead we consider $\Val:\R_{\ge 0}\to\Rmininf$ defined by the semiring homomorphism
\[
\Val(r)=\begin{cases}0 &\text{if } r\ne 0,\\
\infty &\text{if } r=0,\end{cases}
\] 
then for $\Val$ to be continuous, the topology on $\R_{\ge 0}$ extending the Euclidean topology on $\R_{>0}$ must satisfy that any sequence that converges to $0$ stabilises at $0$. But then for   $u\in\ToeplitzU^{res}_\infty(\R_{>0})$ the limits $\lim_{k\to\infty}m_{ik}$ and $\lim_{k\to\infty}m_{kj}$ cannot be $0$, since the $m_{ij}$ coordinates are in $\R_{>0}$. Thus now Proposition~\ref{p:Edreilimits} implies that the Schoenberg parameters of such a Toeplitz matrix $u$ must all be nonzero and
\[\ToeplitzU_\infty^{res}(\R_{>0})=\ToeplitzU_\infty^{\circ}(\R_{>0}).
\] 
The associated tropicalisation nevertheless again just contains the $0$ min-ideal filling.
\end{example}

\begin{example}\label{ex:usualR}
If we consider the generalised Puiseaux series $\KK$ with its $t$-adic topology, then $(\KK_{>0},\Val)$ is another example to which we can apply the construction from Definition~\ref{d:Toeprest}.  The convergence condition on the standard coordinates $(m_{ij})_{i,j}$ of $u\in\ToeplitzU^{res}_\infty(\KK_{>0})$ in this case implies that the valuations $M_{ij}$ must either diverge or stabilise along rows and columns. Moreover, using the min-ideal condition, if one row diverges to $\infty$, so do all subsequent rows, and similarly for columns. Thus we have described a restrictive subset of $\Imin_{\infty}(\Rmin)$ that the tropicalisation $\ToeplitzU^{res}_{\infty,\KK_{>0}}(\Rmin)$ must be contained in. 
\end{example}

We now introduce a choice of $(\RR_{>0},\Val)$ to replace $(\KK_{>0},\Val)$, which is more flexible and better suited for tropicalisation in the infinite setting.

\begin{defn}[A positive semifield of continuous functions]\label{ex:continuous}
Fix $0<\delta<1$ and let $\Kdel:=C^0((0,\delta])$ be the ring of continuous $\R$-valued functions on $(0,\delta]$. Set
\begin{equation}\label{e:Kdelpos}
\begin{array}{ccl}
\Kdelpos&:=&\left\{f:(0,\delta]\to \R_{> 0}\mid \text{ $f$ is continuous and $\exists \ F\in\R$ with $\lim_{t\to 0} t^{-F}f(t)\in\R_{>0}$}\right\}.
\end{array}
\end{equation}
%We also let $\Kdelz=\Kdelpos\cup\{0\}$. 
Note that $\Kdel$ itself is not a valuation ring, and not an integral domain. But we do have a valuation $\Val:\Kdelnn\to\Rmininf$ given by sending $f\in\Kdelpos$ to its associated $F$ from \eqref{e:Kdelpos}, and $0$ to $\infty$. We  also consider the subsemiring $\mathcal O_{>0}:=\{f\in\Kdelpos\mid \Val(f)\ge 0\}$, which we can think of as lying inside $C^0([0,\delta])$.  

Since neither of the familiar topologies of pointwise convergence or uniform convergence make $\Val$ continuous, we replace these by two new topologies. Let us use the supremum norm $\|\ \|_{\sup{}}$ on $\mathcal O_{>0}$ to define the translated $\ep$-balls, $\mathcal U_{A,f,\ep}:=t^A B_\ep(f)$, where $A\in\R$, $f\in \mathcal O_{>0}$ and $B_\ep(f)=\{g\in\Odelpos \mid \|g-f\|_{\sup}<\ep\}$.
\begin{itemize}
\item The \textit{strong topology} on $\Kdelnn$ is the topology  generated by the translated $\ep$-balls $\mathcal U_{A,f,\ep}=t^AB_\ep(f)$  together with  $\{0\}$.
\item The \textit{weak topology} is the topology combining pointwise convergence on $(0,\delta]$ with convergence of the valuations. 
\end{itemize}
Let $\Kdelposst$ and $\Kdelposwk$ signify $\Kdelpos$ endowed with the strong or weak topology, respectively, and similarly for $\Kdelnn$.\end{defn}

\subsection{Two tropical versions of the Edrei Theorem} In order to tropicalise the totally positive Toeplitz matrices we need only to construct $\RR_{>0}$-valued points in a big enough open subset, and we make the ansatz $\gamma=0$. Let us suppose $u(\mathbf c)\in \ToeplitzU_{\infty}(\RR_{>0})$ has a generating function of the form
\begin{equation}\label{e:SchoenbergFormulaNoGammaFormal}
1+c_1x+c_2x^2+c_3x^3+\dotsc =\prod_{i=1}^\infty\frac{1+\boldbeta_ix}{1-\boldalpha_ix}.
\end{equation}
Then by a (supersymmetric) Jacobi-Trudi formula  Toeplitz minors of $u(\mathbf c)$ have an expression as  supersymmetric Schur functions $S_\lambda(\boldalpha||\boldbeta)$. See \cite{BalantekinBars,DondiJarvis,BereleRegev} as well as \cite{Macdonald:Book1} for more general background. For example, for 
the $k$-th coefficient
\[
c_k=S_{(k)}(\boldalpha||\boldbeta)=\sum_{i+j=k} h_i(\boldalpha)e_j(\boldbeta),\] 
where $h_i$ and $e_j$ are the complete and elementary symmetric functions, and 
%Furthermore, the minors appearing in the formula \eqref{e:mviaminors} for the standard coordinates $m_{ij}$ of $u(\mathbf c)$ are then given by 
\[
\Minor^{[i]}_{[i]+j}(u(\mathbf c)) = S_{i\times j}(\bold a||\bold b),
\]
where $i\times j$ denote the partition $j^i$, whose Young diagram is an $i\times j$ rectangle.

\begin{defn}\label{d:Omegas}
Write $\boldalpha=(\boldalpha_1,\boldalpha_2,\dotsc)$ and $\boldbeta=(\boldbeta_1,\boldbeta_2,\dotsc)$ for two infinite sequences in $\RR_{\ge 0}$ and set 
\begin{equation}\label{e:mij(a,b)}
\bold m_{ij}(\boldalpha,\boldbeta):=\frac{S_{i\times j}(\boldalpha||\boldbeta)S_{{(i-1)}\times{(j-1)}}(\boldalpha||\boldbeta)}{S_{i\times{(j-1)}}(\boldalpha||\boldbeta)S_{{(i-1)}\times j}(\boldalpha||\boldbeta)}.
\end{equation}
We define the (restricted) generalised Schoenberg parameter space $\Omega^{res}(\RR_{>0})$ as follows,
\begin{equation*}
\Omega^{res}(\RR_{>0}):=\left\{(\boldalpha,\boldbeta)\in\RR_{\ge 0}^\N\times \RR_{\ge 0}^\N\left |\begin{array}{l} \boldalpha_1\ge\boldalpha_2\ge\boldalpha_3\ge\dotsc, \\
\boldbeta_1\ge\boldbeta_2\ge\boldbeta_3\ge \dotsc, \\
\boldbeta_i+\boldalpha_i\ne 0 \text{ for all $i\in\N,$}\\
  \text{$S_{i\times j}(\boldalpha||\boldbeta)$ is well-defined (converges) in $\RR_{>0}$ for all $i,j\in\N$,}\\
  \exists
   \text{ $\lim_{k\to\infty}\mathbf m_{ik}(\boldalpha,\boldbeta)$ and $\lim_{k\to\infty}\mathbf m_{kj}(\boldalpha,\boldbeta)$  in $\RR_{\ge 0}$,  for all $k\in\N$}
\end{array}
\right.\right\}.
\end{equation*}
Then we have the following well-defined map 
\begin{eqnarray}\label{e:Tres}
\mathcal T^{res}:\Omega^{res}(\RR_{>0}) &\longrightarrow & \Toeplitz^{res}_\infty(\RR_{>0})\\
(\boldalpha,\boldbeta)&\mapsto & u(\bold c),\notag
\end{eqnarray}
where $u(\bold c)$ is as in \eqref{e:u}, and $\bold c=(c_i)_{i\in \N}$ is determined by 
\begin{equation}\label{e:SchoenbergFormulaNoGamma}
1+c_1x+c_2x^2+c_3x^3+\dotsc =\prod_{i=1}^\infty\frac{1+\boldbeta_jx}{1-\boldalpha_ix}.
\end{equation}
 \end{defn}
We will now formulate two infinite analogues of the tropical parametrisation theorem, Theorem~\ref{t:IntroTropFinite}. Both are modelled on the Edrei theorem, but they involve different choices for $\RR_{>0}$, namely $\Kdelposst$ and $\Kdelposwk$.
%, for the tropicalisation. 
\begin{defn}\label{d:interlacing}
Let $\bigA=(A_1,A_2,\dotsc)$ and $\bigB=(B_1,B_2,\dotsc)$ be two  infinite sequences in $\R\cup\{\infty\}$. 
\begin{itemize}
\item We say that $\bigA$ and $\bigB$ are \textit{weakly interlacing} if $\bigA$ and $\bigB$ are weakly increasing with the same supremum in $\R\cup\{\infty\}$.
\item
If $\bigA$ and $\bigB$ are weakly increasing sequences such that for every $A_i$ there exists a $B_j$ with $B_j\ge A_i$ and for every $B_j$ there exists an $A_i$ with $A_i\ge B_j$, then we call $\bigA$ and $\bigB$ \textit{interlacing}.   
\end{itemize}
We define the following tropical parameter spaces
\begin{eqnarray*}
\Omega_\star(\Rmin)&:=&\{(\bigA,\bigB)\mid \text{$\bigA$ and $\bigB$ are weakly interlacing, $\min(A_k,B_k)\in\R$}\},
\\	  
\Omega^{il}_\star(\Rmin)&:=&\{(\bigA,\bigB)\mid \text{$\bigA$ and $\bigB$ are interlacing, $\min(A_k,B_k)\in\R$}\}.
\end{eqnarray*}
\end{defn}
\begin{remark}
Note that interlacing implies weakly interlacing, but not vice versa. For example, if $A_i=\frac{i-1}i$ and all $B_j=1$ then $\bigA$ and $\bigB$ are weakly interlacing but not interlacing.
\end{remark}
The following theorem, which is the first half of \cite[Theorem 10.6]{rietschToeplitz}, says that the Schoenberg type map $\mathcal T^{res}:\Omega^{res}(\Kdelposwk)\to\ToeplitzU^{res}_\infty(\Kdelposwk)$ tropicalises to give an explicit parametrisation of the full  set $\Imin_{\infty}(\Rmin)$ of infinite min-ideal fillings by $\Omega_\star(\Rmin)$. 

\begin{theorem}[$\Kdelposwk$-tropicalised Edrei Theorem]\label{t:restrtropparamWk}
For the $\Kdelposwk$-tropicalised Toeplitz matrices we have that
\[
\ToeplitzU_{\infty,\Kdelposwk}^{res}(\Rmin)=\Imin_{\infty}(\Rmin),
\]
and the following diagram commutes
\begin{equation}\label{e:TropEdreiWk}
\begin{array}{c}
% https://q.uiver.app/#q=WzAsMTAsWzEsMCwiXFxPbWVnYV9cXHN0YXIoXFxLZGVscG9zd2spIl0sWzMsMCwiXFxUb2VwbGl0elVee3Jlc31fXFxpbmZ0eShcXEtkZWxwb3N3aykiXSxbMSwxLCJcXE9tZWdhX1xcc3RhcihcXFJtaW4pIl0sWzMsMSwiXFxJbWluX1xcaW5mdHkoXFxSbWluKSJdLFs0LDAsIihtX3tpan0pX3tpLGp9Il0sWzQsMSwiKE1fe2lqfSlfe2ksan0iXSxbMCwwLCIoXFxib2xkYWxwaGEsXFxib2xkYmV0YSkiXSxbMCwxLCIoXFxiaWdBLFxcYmlnQikiXSxbMSwyLCJcXGJlZ2lue2FycmF5fXtjfXsoXFxiaWdBLFxcYmlnQil9XFxcXHtcXCBcXCB9XFxlbmR7YXJyYXl9Il0sWzMsMiwiXFxiZWdpbnthcnJheX17Y317KFxcbWluKEFfaSxCX2opKV97aSxqXFxpblxcTn19XFxcXHtcXCBcXCB9XFxlbmR7YXJyYXl9Il0sWzAsMV0sWzIsMywiXFxtYXRoYmIgRV9cXHN0YXIiXSxbMCwyLCJcXFZhbCIsMl0sWzEsM10sWzQsNSwiIiwxLHsic3R5bGUiOnsidGFpbCI6eyJuYW1lIjoibWFwcyB0byJ9fX1dLFs2LDcsIiIsMCx7InN0eWxlIjp7InRhaWwiOnsibmFtZSI6Im1hcHMgdG8ifX19XSxbOCw5LCIiLDAseyJvZmZzZXQiOi0yLCJzdHlsZSI6eyJ0YWlsIjp7Im5hbWUiOiJtYXBzIHRvIn19fV1d
\begin{tikzcd}
	{(\boldalpha,\boldbeta)} & {\Omega^{res}(\Kdelposwk)} && {\ToeplitzU^{res}_\infty(\Kdelposwk)} & {(m_{ij})_{i,j}} \\
	{(\bigA,\bigB)} & {\Omega_\star(\Rmin)} && {\Imin_\infty(\Rmin)} & {(M_{ij})_{i,j}} 
	\arrow[maps to, from=1-1, to=2-1]
	\arrow["\mathcal T^{res}", from=1-2, to=1-4]
	\arrow["\Val",two heads, from=1-2, to=2-2]
	\arrow["\Val", two heads, from=1-4, to=2-4]
	\arrow[maps to, from=1-5, to=2-5]
	\arrow["{\mathbb E}", from=2-2, to=2-4]
\end{tikzcd}\\
% https://q.uiver.app/#q=WzAsMixbMCwwLCIoXFxiaWdBLFxcYmlnQikiXSxbMiwwLCIoXFxtaW4oQV9pLEJfaikpX3tpLGpcXGluXFxOfSJdLFswLDEsIiIsMCx7InN0eWxlIjp7InRhaWwiOnsibmFtZSI6Im1hcHMgdG8ifX19XV0=
\begin{tikzcd}
	{(\bigA,\bigB)} && {(\min(A_i,B_j))_{i,j\in\N}.}
	\arrow[maps to, from=1-1, to=1-3]
\end{tikzcd}
\end{array}
\end{equation}
Moreover the tropicalisation $\mathbb E$ of $\mathcal T^{res}$ is a bijection.
\end{theorem}
We remark that a key ingredient for the commutativity of this diagram is a tableau formula for  supersymmetric Schur functions due to Goulden-Greene and Macdonald  \cite{GouldenGreene94,Macdonald:variations}. We now  state the second half of \cite[Theorem~10.6]{rietschToeplitz}.  Recall the set of stable min-ideal fillings $\Imin_{\infty}^s(\Rmin)$ from Definition~\ref{d:minIdeal}.

\begin{theorem}[$\Kdelposst$-tropicalised Edrei Theorem]\label{t:restropparamSt}
For the $\Kdelposst$-tropicalised Toeplitz matrices we have that
\[
\ToeplitzU_{\infty,\Kdelposst}^{res}(\Rmin)=\Imin^s_{\infty}(\Rmin),
\]
and the commutative diagram~\eqref{e:TropEdreiWk} restricts to 
\begin{equation}\label{e:restrOmVal}
% https://q.uiver.app/#q=WzAsNCxbMCwwLCJcXE9tZWdhXntyZXN9KFxcS2RlbHBvc3drKSJdLFsyLDAsIlxcVG9lcGxpdHpVXntyZXN9X1xcaW5mdHkoXFxLZGVscG9zd2spIl0sWzAsMSwiXFxPbWVnYV9cXHN0YXIoXFxSbWluKSJdLFsyLDEsIlxcSW1pbl9cXGluZnR5Il0sWzAsMV0sWzIsMywiXFxtYXRoYmIgRV9cXHN0YXIiXSxbMCwyLCJcXFZhbCIsMl0sWzEsM11d
\begin{tikzcd}
	{\Omega^{res}(\Kdelposst)} && {\ToeplitzU^{res}_\infty(\Kdelposst)} \\
	{\Omega_\star^{il}(\Rmin)} && {\Imin^{s}_\infty}(\Rmin).
	\arrow["{\mathcal T^{res}}", from=1-1, to=1-3]
	\arrow["\Val",two heads, from=1-1, to=2-1]
	\arrow["\Val",two heads, from=1-3, to=2-3]
	\arrow["{\mathbb E^s}", from=2-1, to=2-3]
\end{tikzcd}
\end{equation}
Moreover, the map $\mathbb E^s$ is a bijection. 
\end{theorem}

The `restricted' condition from Definition~\ref{d:Toeprest} ensures that the tropicalisation of the map $\mathcal T^{res}$ becomes a bijection in the above two theorems. Namely, this condition leads directly to the (weakly) interlacing condition in the parameter spaces in Definition~\ref{d:interlacing}. Without these conditions the map $\mathbb E$ still makes sense, but it would no longer be injective, compare~\cite[Lemma 5.16 and Corollary 9.12]{rietschToeplitz}.  We call the parameters $A_i, B_j$ appearing in Theorem~\ref{t:restrtropparamWk} and Theorem~\ref{t:restropparamSt} the \textit{tropical Schoenberg parameters}.
 
\subsection{Asymptotics of Lusztig weights} 
To state our final asymptotic result 
%which is a version of \cite[Theorem 6.1]{rietschToeplitz} 
we need one more type of min-ideal filling, which is the tropical analogue of $\ToeplitzU_\infty^{\circ}$. Namely,
\[
\Imin^{\mathbb R}_\infty(\Rmin)=\{(M_{ij})_{i,j}\in\Imin_{\infty}(\Rmin)\mid \lim_{k\to\infty}M_{ik},\ \lim_{k\to\infty}M_{kj}\in \R\}.
\]
We call min-ideal fillings in $\Imin^{\mathbb R}_\infty(\Rmin)$ asymptotically real. 
Theorem~\ref{t:VKtypethmtrop} below, which is a version of  \cite[Theorem 6.1]{rietschToeplitz}, has a proof that is  analogous to the one of Theorem~\ref{t:IntroVKtyp}, but with Proposition~\ref{p:Edreilimits} replaced by an argument from \cite[Theorem~5.8]{rietschToeplitz}.  However, Theorem~\ref{t:VKtypethmtrop} preceded and motivated Theorem~\ref{t:IntroVKtyp}.  

\begin{theorem}[{}]\label{t:VKtypethmtrop}
Suppose $(M^{(n+1)})_n$ is a sequence of min-ideal fillings with $M^{(n+1)}\in\Imin_{n+1}$, such that $(M^{(n+1)})_n$ converges uniformly to %an asymptotically real min-ideal filling 
$M^{(\infty)}=(M_{ij})_{i,j\in\N}\in\Imin^\R_\infty(\Rmin)$. Let $(\bigA,\bigB)$ be the tropical Schoenberg parameters for $M^{(\infty)}$ from Theorem~\ref{t:restrtropparamWk}, and 
% We have $M^{(\infty)}=\mathbb E(\bigA,\bigB)$ for real parameter sequences $\bigA=(A_i)_i$ and $\bigB=(B_j)_j$ with $\mathbb E$ as in \eqref{e:TropEdreiWk}. 
let the Lusztig weight of $M^{(n+1)}$ be given by
\begin{equation}\label{e:Ln+1}
\mathcal L^{(n+1)}(M^{(n+1)})=\begin{pmatrix}
\lambda^{(n+1)}_1 & &&&\\
&\lambda^{(n+1)}_2 & &&\\
&&\ddots &&\\
&&&\lambda^{(n+1)}_n & \\
&&&&\lambda^{(n+1)}_{n+1}
\end{pmatrix}.
\end{equation}
Then the normalised limits of the finite parameters are real and recover the tropical Schoenberg parameters,
\begin{equation}\label{e:limits}
\lim_{n\to\infty}\frac{\lambda^{(n+1)}_i}n=A_i\qquad \text{and}\qquad \lim_{n\to\infty}\frac{-\lambda^{(n+1)}_{n+2-j}}{n}=B_j.
\end{equation}
\end{theorem}

\begin{remark}
The asymptotically real condition is necessary for the theorem to hold \cite[Remark 6.3]{rietschToeplitz}.
\end{remark}

In Theorem~\ref{t:VKtypethmtrop} the relation between the $A_i$ and the $B_i$ gains an interpretation as the limit of the natural involution $-w_0$ on $\mathfrak{h}_{PSL_{n+1}}^*$ which relates the weights of dual representations. 
Thus the Lusztig weight map $\mathcal L^{(n+1)}$ and its dual $\mathcal L^{(n+1),\vee}:=-w_0\circ \mathcal L^{(n+1)}$, after normalisation, restrict to give a map
\begin{equation}\label{e:FinEdreiInvAnalogue}
\frac 1n(\mathcal L^{(n+1)},\mathcal L^{(n+1),\vee}):\Imin_{n+1}\to \mathfrak h_{\R,PSL_{n+1}}^*\oplus \mathfrak h_{\R,PSL_{n+1}}^*,
\end{equation}
which can be considered a finite analogue of the restriction of $\mathbb E\inv$,
\begin{equation}\label{e:Edreiinv}
(\mathbb E^\R)\inv: \Imin^\R_{\infty}(\Rmin)\ \overset\sim\longrightarrow\ \Omega^{\R}_\star(\Rmin),
\end{equation}
compare   \cite[Theorem~5.8]{rietschToeplitz}. We observe that $\frac{1}n\mathcal L^{(n+1)}$ and 
$\frac{1}n\mathcal L^{(n+1)}$ individually have image all of $\mathfrak h_{\R,PSL_{n+1}}^*$, while the image of $(\mathbb E^\R)\inv$ is made up  of infinite  `antidominant' sequences (since the $A_i$ and $B_j$ are weakly increasing). Meanwhile the image of the combined map \eqref{e:FinEdreiInvAnalogue} lies in an `antidiagonal'  subspace $\{(h,- w_0\cdot h)\mid h\in\mathfrak h^*_{\R,PSL_{n+1}}\}$, while the only relationship between the components of $(\bigA,\bigB)$ in the image of $(\mathbb E^\R)\inv$ is that they have the same supremum.

\subsection*{Acknowledgements} It is a pleasure to thank George Lusztig, whose work on total positivity has been very inspiring to me, and  Dale Peterson, who sparked my interest in Toeplitz matrices and quantum cohomology. I am very grateful to my former students Jamie Judd and Teresa L\"udenbach for their beautiful work, and  to all of my collaborators, particularly Bethany Marsh, Lauren Williams, Clelia Pech,  Changzheng Li and Mingzhi Yang, for the multitude of flag variety-related projects and discussions. 
I also thank Alexei Borodin, Yang-Hui He, Alexander Pushnitski and Nicholas Shepherd-Barron for helpful conversations related to this work.

\bibliographystyle{plain}%
\bibliography{biblio}

\def\cprime{$'$}
\begin{thebibliography}{10}

\bibitem{ASW}
Michael Aissen, I.~J. Schoenberg, and A.~M. Whitney.
\newblock On the generating functions of totally positive sequences. {I}.
\newblock {\em J. Anal. Math.}, 2:93--103, 1952.

\bibitem{AF:book}
David Anderson and William Fulton.
\newblock {\em Equivariant cohomology in algebraic geometry}, volume 210 of
  {\em Cambridge Studies in Advanced Mathematics}.
\newblock Cambridge University Press, Cambridge, 2024.

\bibitem{BalantekinBars}
A.~B. Balantekin and I.~Bars.
\newblock Representations of supergroups.
\newblock {\em J. Math. Phys.}, 22(8):1810--1818, 1981.

\bibitem{BereleRegev}
A~Berele and A~Regev.
\newblock Hook {Y}oung diagrams with applications to combinatorics and to
  representations of {L}ie superalgebras.
\newblock {\em Advances in Mathematics}, 64(2):118--175, 1987.

\bibitem{BeKa}
Arkady Berenstein and David Kazhdan.
\newblock Lecture notes on geometric crystals and their combinatorial
  analogues.
\newblock In {\em Combinatorial aspect of integrable systems}, volume~17 of
  {\em MSJ Mem.}, pages 1--9. Math. Soc. Japan, Tokyo, 2007.

\bibitem{BorgerGrinberg}
James Borger and Darij Grinberg.
\newblock Boolean {W}itt vectors and an integral {E}drei-{T}homa theorem.
\newblock {\em Selecta Math. (N.S.)}, 22(2):595--629, 2016.

\bibitem{BorodinOlshanskiBook}
Alexei Borodin and Grigori Olshanski.
\newblock {\em Representations of the infinite symmetric group}, volume 160 of
  {\em Camb. Stud. Adv. Math.}
\newblock Cambridge: Cambridge University Press, 2016.

\bibitem{BossingerCMN24}
Lara Bossinger, Man-Wai Cheung, Timothy Magee, and Alfredo {Nájera Chávez}.
\newblock Newton–{O}kounkov bodies and minimal models for cluster varieties.
\newblock {\em Advances in Mathematics}, 447:109680, 2024.

\bibitem{Brion}
Michel Brion.
\newblock Lectures on the geometry of flag varieties.
\newblock In {\em Topics in cohomological studies of algebraic varieties},
  Trends Math., pages 33--85. Birkh\"auser, Basel, 2005.

\bibitem{BufetovGorinThoma}
Alexey Bufetov and Vadim Gorin.
\newblock Stochastic monotonicity in {Y}oung graph and {T}homa theorem.
\newblock {\em Int. Math. Res. Not. IMRN}, (23):12920--12940, 2015.

\bibitem{Chhaibi:Thesis}
Reda Chhaibi.
\newblock Littelmann path model for geometric crystals, {W}hittaker functions
  on {L}ie groups and {B}rownian motion.
\newblock PhD thesis, 2013.

\bibitem{Chow:Flag}
C.~H. Chow.
\newblock On {D}. {P}eterson's presentation of quantum cohomology of {$G/P$}.
\newblock arXiv:2210.17382, 2024.

\bibitem{Chow:FlagDmirror}
C.~H. Chow.
\newblock The {$D_{\hslash}$}-module mirror conjecture for flag varieties.
\newblock arXiv:2311.15523, 2025.

\bibitem{Chow:Gamma}
C.~H. Chow.
\newblock Gamma conjecture {I} for flag varieties.
\newblock arXiv:2501.13221, 2025.

\bibitem{Chow:PLSTheorem}
C.~H. Chow.
\newblock Peterson–{L}am–{S}himozono’s theorem is an affine analogue of
  the quantum {C}hevalley formula.
\newblock {\em Forum of Mathematics, Sigma}, 13:e62, 2025.

\bibitem{C-F:QCohFl}
Ionu{\c{t}} Ciocan-Fontanine.
\newblock Quantum cohomology of flag varieties.
\newblock {\em IMRN}, 1995(6):263--277, 1995.

\bibitem{Davydov}
A.~A. Davydov.
\newblock Totally positive sequences and {$R$}-matrix quadratic algebras.
\newblock volume 100, pages 1871--1876. 2000.
\newblock Algebra, 12.

\bibitem{DondiJarvis}
P.~H. Dondi and P.~D. Jarvis.
\newblock Diagram and superfield techniques in the classical superalgebras.
\newblock {\em J. Phys. A, Math. Gen.}, 14:547--563, 1981.

\bibitem{Edrei52}
A.~Edrei.
\newblock On the generating functions of totally positive sequences {II}.
\newblock {\em J. Analyse Math.}, 2:104--109, 1952.

\bibitem{EdreiDouble}
Albert Edrei.
\newblock On the generating function of a doubly infinite totally positive
  sequence.
\newblock {\em Transactions of the American Mathematical Society},
  74(3):367--383, 1953.

\bibitem{Edrei53}
Albert Edrei.
\newblock Proof of a conjecture of {S}choenberg on the generating function of a
  totally positive sequence.
\newblock {\em Canadian Journal of Mathematics}, 5:86–94, 1953.

\bibitem{FGP}
Sergey Fomin, Sergei Gelfand, and Alexander Postnikov.
\newblock Quantum {Schubert} polynomials.
\newblock {\em J. Am. Math. Soc.}, 10(3):565--596, 1997.

\bibitem{FZ:DoubleBruhat}
Sergey Fomin and Andrei Zelevinsky.
\newblock Double {B}ruhat cells and total positivity.
\newblock {\em J. Amer. Math. Soc.}, 12(2):335--380, 1999.

\bibitem{GK:oscillation}
F.~P. Gantmacher and M.~G. Krein.
\newblock {\em Oscillation matrices and kernels and small vibrations of
  mechanical systems}.
\newblock AMS Chelsea Publishing, Providence, RI, revised edition, 2002.
\newblock Translation based on the 1941 Russian original, Edited and with a
  preface by Alex Eremenko.

\bibitem{Gin:GS}
V.~Ginzburg.
\newblock {Perverse sheaves on a {L}oop group and {L}anglands' duality}.
\newblock arXiv:9511007, 1995.

\bibitem{GiventalKim:TodaFl}
Alexander Givental and Bumsig Kim.
\newblock Quantum cohomology of flag manifolds and {Toda} lattices.
\newblock {\em Commun. Math. Phys.}, 168(3):609--641, 1995.

\bibitem{Givental:QToda}
Alexander~B. Givental.
\newblock Stationary phase integrals, quantum {T}oda lattices, flag manifolds
  and the mirror conjecture.
\newblock In {\em Topics in singularity theory}, volume 180 of {\em Amer. Math.
  Soc. Transl. Ser. 2}, pages 103--115. Amer. Math. Soc., Providence, RI, 1997.

\bibitem{GohmKostler12}
Rolf Gohm and Claus K\"ostler.
\newblock Noncommutative independence in the infinite braid and symmetric
  group.
\newblock In {\em Noncommutative harmonic analysis with applications to
  probability {III}}, volume~96 of {\em Banach Center Publ.}, pages 193--206.
  Polish Acad. Sci. Inst. Math., Warsaw, 2012.

\bibitem{GouldenGreene94}
Ian Goulden and Curtis Greene.
\newblock A new tableau representation for supersymmetric {S}chur functions.
\newblock {\em J. Algebra}, 170(2):687--703, 1994.

\bibitem{Grochenig23}
Karlheinz Gr{\"o}chenig.
\newblock {\em Schoenberg's Theory of Totally Positive Functions and the
  Riemann Zeta Function}, pages 193--210.
\newblock Springer International Publishing, Cham, 2023.

\bibitem{GHKK}
Mark {Gross}, Paul {Hacking}, Sean {Keel}, and Maxim {Kontsevich}.
\newblock Canonical bases for cluster algebras.
\newblock {\em J. Amer. Math. Soc.}, 31(2):497--608, 2018.

\bibitem{HoHai}
Phung~Ho Hai.
\newblock Poincar\'e{} series of quantum spaces associated to {H}ecke
  operators.
\newblock {\em Acta Math. Vietnam.}, 24(2):235--246, 1999.

\bibitem{HiraiHirai}
T.~Hirai and E.~Hirai.
\newblock Characters for the infinite {W}eyl groups of type
  {$B_\infty/C_\infty$} and {$D_\infty$}, and for analogous groups.
\newblock In {\em Non-commutativity, infinite-dimensionality and probability at
  the crossroads}, volume~16 of {\em QP--PQ: Quantum Probab. White Noise
  Anal.}, pages 296--317. World Sci. Publ., 2002.

\bibitem{Humphreys:LAG}
James~E. Humphreys.
\newblock {\em Linear algebraic groups}, volume No. 21 of {\em Graduate Texts
  in Mathematics}.
\newblock Springer-Verlag, New York-Heidelberg, 1975.

\bibitem{Johnston}
S.~Johnston.
\newblock Quantum periods, toric degenerations and intrinsic mirror symmetry.
\newblock arXiv:2501.01408, 2025.

\bibitem{Judd:Flag}
J.~Judd.
\newblock Tropical critical points of the superpotential of a flag variety.
\newblock {\em J. Algebra}, 497:102--142, 2018.

\bibitem{JuddRietsch:24}
J.~Judd and K.~Rietsch.
\newblock A positive/tropical critical point theorem and mirror symmetry.
\newblock {\em Adv. Math.}, 456:63, 2024.
\newblock Id/No 109911.

\bibitem{Karlin}
S.~Karlin.
\newblock Total positivity. {Volume} {I}.
\newblock {Stanford} {University} {Press}. {XIV}, 576 p., 1968.

\bibitem{Kerov-thesis}
S.~V. Kerov.
\newblock {\em Asymptotic representation theory of the symmetric group and its
  applications in analysis}, volume 219 of {\em Transl. Math. Monogr.}
\newblock Providence, RI: American Mathematical Society (AMS), 2003.

\bibitem{KOO-IMRN}
Sergei Kerov, Andrei Okounkov, and Grigori Olshanski.
\newblock The boundary of the {Y}oung graph with {J}ack edge multiplicities.
\newblock {\em Internat. Math. Res. Notices}, (4):173--199, 1998.

\bibitem{Kos:QCoh}
Bertram Kostant.
\newblock Flag manifold quantum cohomology, the {Toda} lattice, and the
  representation with highest weight {{\(\rho\)}}.
\newblock {\em Sel. Math., New Ser.}, 2(1):43--91, 1996.

\bibitem{LP:I}
Thomas Lam and Pavlo Pylyavskyy.
\newblock Total positivity in loop groups, {I}: {W}hirls and curls.
\newblock {\em Adv. Math.}, 230(3):1222--1271, 2012.

\bibitem{LP:II}
Thomas Lam and Pavlo Pylyavskyy.
\newblock Total positivity for loop groups {II}: {C}hevalley generators.
\newblock {\em Transform. Groups}, 18(1):179--231, 2013.

\bibitem{LamRietsch}
Thomas Lam and Konstanze Rietsch.
\newblock Total positivity, {Schubert} positivity, and geometric {Satake}.
\newblock {\em J. Algebra}, 460:284--319, 2016.

\bibitem{LamShimozono:2010}
Thomas Lam and Mark Shimozono.
\newblock Quantum cohomology of {$G/P$} and homology of affine {G}rassmannian.
\newblock {\em Acta Mathematica}, 204(1):49--90, 2010.

\bibitem{LamTemplier}
Thomas Lam and Nicolas Templier.
\newblock The mirror conjecture for minuscule flag varieties.
\newblock {\em Duke Math. J.}, 173(1):75--175, 2024.

\bibitem{LasSch:Schubert}
Alain Lascoux and Marcel-Paul Sch\"utzenberger.
\newblock Polyn\^omes de {S}chubert.
\newblock {\em C. R. Acad. Sci. Paris S\'er. I Math.}, 294(13):447--450, 1982.

\bibitem{LRYZ}
Changzheng Li, Konstanze Rietsch, Mingzhi Yang, and Chi Zhang.
\newblock A {P}l\"ucker coordinate mirror for partial flag varieties and
  quantum {S}chubert calculus.
\newblock arXiv:2401.15640[math.ag], 2024.

\bibitem{Ludenbach}
Teresa L{\"u}denbach.
\newblock Ideal polytopes for representations of {$GL_n(C)$}.
\newblock {arXiv:2206.10522 [math.RT]}, 2022.

\bibitem{Lusztig:cbJAMS}
G.~Lusztig.
\newblock Canonical bases arising from quantized enveloping algebras.
\newblock {\em J. Amer. Math. Soc.}, 3(2):447--498, 1990.

\bibitem{Lusztig94}
G.~Lusztig.
\newblock Total positivity in reductive groups.
\newblock In {\em Lie theory and geometry}, volume 123 of {\em Progr. Math.},
  pages 531--568. Birkh\"auser Boston, Boston, MA, 1994.

\bibitem{Macdonald:variations}
I.~G. Macdonald.
\newblock Schur functions: theme and variations.
\newblock In {\em S\'{e}minaire {L}otharingien de {C}ombinatoire
  ({S}aint-{N}abor, 1992)}, volume 498 of {\em Publ. Inst. Rech. Math. Av.},
  pages 5--39. Univ. Louis Pasteur, Strasbourg, 1992.

\bibitem{Macdonald:Book1}
I.~G. Macdonald.
\newblock {\em Symmetric functions and {Hall} polynomials. {With} contributions
  by {A}. {V}. {Zelevinsky}}.
\newblock Oxford: Oxford University Press, reprint of the 1998 2nd edition
  edition, 2015.

\bibitem{Markwig:genPuiseux}
T.~Markwig.
\newblock A field of generalised {P}uiseux series for tropical geometry.
\newblock {\em Rend. Semin. Mat. Univ. Politec. Torino}, 68(1):79--92, 2010.

\bibitem{MR}
B.~R. Marsh and K.~Rietsch.
\newblock The {$B$}-model connection and mirror symmetry for {G}rassmannians.
\newblock {\em Adv. Math.}, 366:107027, 131, 2020.

\bibitem{OkInfSymmGrp}
A.~Okounkov.
\newblock On representations of the infinite symmetric group.
\newblock {\em J. Math. Sci., NY}, 96(5):1, 1997.

\bibitem{OkThoma}
A.~Yu. Okounkov.
\newblock Thoma's theorem and representations of an infinite bisymmetric group.
\newblock {\em Funktsional. Anal. i Prilozhen.}, 28(2):31--40, 95, 1994.

\bibitem{Olshanskii90}
G.~I. Olshanski.
\newblock Unitary representations of {$(G,K)$}-pairs that are connected with
  the infinite symmetric group {$S(\infty)$}.
\newblock {\em Algebra i Analiz}, 1(4):178--209, 1989.

\bibitem{peterson}
D.~Peterson.
\newblock {Quantum cohomology of $G/P$}.
\newblock {Lecture Course, MIT, Spring Term}, 1997.

\bibitem{PetrovThoma}
F.~V. Petrov.
\newblock Asymptotics of traces of paths on {Y}oung and {S}chur graphs.
\newblock {\em Zap. Nauchn. Sem. S.-Peterburg. Otdel. Mat. Inst. Steklov.
  (POMI)}, 468:126--137, 2018.
\newblock Introduction by A. M. Vershik.

\bibitem{rietschJAMS}
K.~Rietsch.
\newblock Totally positive {Toeplitz} matrices and quantum cohomology of
  partial flag varieties.
\newblock {\em J. Am. Math. Soc.}, 16(2):363--392, 2003.

\bibitem{rietschNagoya}
K.~Rietsch.
\newblock A mirror construction for the totally nonnegative part of the
  {P}eterson variety.
\newblock {\em Nagoya Math. J.}, 183:105--142, 2006.

\bibitem{rietsch}
K.~Rietsch.
\newblock A mirror symmetric construction of {$qH^\ast_T(G/P)_{(q)}$}.
\newblock {\em Adv. Math.}, 217(6):2401--2442, 2008.

\bibitem{rietschToeplitz}
K.~Rietsch.
\newblock Tropical {T}oeplitz matrices and parametrisations.
\newblock arXiv:2509.06944[math.rt], 2025.

\bibitem{RW}
K.~{Rietsch} and L.~{Williams}.
\newblock {Newton-Okounkov} bodies, cluster duality, and mirror symmetry for
  {Grassmannians}.
\newblock {\em Duke Math. J.}, 168(18):3437--3527, 2019.

\bibitem{Schoenberg:48}
I.~J. Schoenberg.
\newblock Some analytical aspects of the problem of smoothing.
\newblock In {\em Studies and {E}ssays {P}resented to {R}. {C}ourant on his
  60th {B}irthday, {J}anuary 8, 1948}, pages 351--370. 1948.

\bibitem{Springer:book}
T.~A. Springer.
\newblock {\em Linear algebraic groups}.
\newblock Modern Birkh\"{a}user Classics. Birkh\"{a}user Boston, Inc., Boston,
  MA, second edition, 2009.

\bibitem{Thoma}
E.~Thoma.
\newblock Die unzerlegbaren, positive–definiten {K}lassenfunktionen der
  abz\"ahlbar un-endlichen, symmetrischen {G}ruppe.
\newblock {\em Math. Zeitschr.}, 85:40--6, 1964.

\bibitem{VershikSurvey}
A.~M. Vershik.
\newblock Three theorems on the uniqueness of the {P}lancherel measure from
  different viewpoints.
\newblock {\em Tr. Mat. Inst. Steklova}, 305:71--85, 2019.

\bibitem{VershikKerovAsymptotic}
A.~M. Vershik and S.~V. Kerov.
\newblock Asymptotic theory of characters of the symmetric group. (english.
  russian original).
\newblock {\em Funct. Anal. Appl.}, 15:246--255, 1982.

\bibitem{VT:SchurWeyl}
A.~M. Vershik and N.~V. Tsilevich.
\newblock The {S}chur-{W}eyl graph and {T}homa's theorem.
\newblock {\em Funct. Anal. Appl.}, 55(3):198--209, 2021.
\newblock Translation of Funktsional. Anal. i Prilozhen. {\bf 55} (2021), no.
  3, 26--41.

\bibitem{VoiculescuClassical}
D.~Voiculescu.
\newblock Sur les repr\'esentations factorielles finies de {$U$} {$(\infty )$}
  et autres groupes semblables.
\newblock {\em C. R. Acad. Sci. Paris S\'er. A}, 279:945--946, 1974.

\bibitem{Voiculescu}
D.~Voiculescu.
\newblock Repr{\'e}sentations factorielles de type {{\(II_1\)}} de
  {{\(U(\infty)\)}}.
\newblock {\em J. Mat. Pur. App.(9)}, 55:1--20, 1976.

\bibitem{YunZhu:Loop}
Zhiwei Yun and Xinwen Zhu.
\newblock Integral homology of loop groups via {L}anglands dual groups.
\newblock {\em Represent. Theory}, 15:347--369, 2011.

\end{thebibliography}

 \end{document}